\documentclass[a4paper,12pt]{amsart}

\usepackage{amsmath,amsthm,amsfonts,amssymb,stmaryrd,mathrsfs,enumitem}
\usepackage{latexsym}
\usepackage{fullpage}
\usepackage{a4,color,palatino,fancyhdr}
\usepackage{graphicx}
\usepackage{float}  
\usepackage[small, bf, margin=90pt, tableposition=bottom]{caption}
\RequirePackage{amssymb}
\RequirePackage[T1]{fontenc}
\usepackage{amscd}
\usepackage{xypic}
\input{xy}
\xyoption{all}
\xyoption{poly}
\usepackage[all]{xy}
\usepackage{tikz}
\newtheorem{theorem}{Theorem}[section]
\newtheorem{proposition}[theorem]{Proposition}

\newtheorem{lemma}[theorem]{Lemma}
\newtheorem{corollary}[theorem]{Corollary}
\newtheorem{remark}{Remark}
\newtheorem{example}{Example}

\numberwithin{equation}{section}

\def\im{\mathop{\rm Im}\nolimits}

\newcommand\Z{{\mathbb{Z}}}
\newcommand\R{{\mathbb{R}}}
\newcommand\C{{\mathbb{C}}}

\begin{document}

\title{Elementary symmetric polynomials in Stanley--Reisner face ring}
\author{Zhi L\"u, Jun Ma and Yi Sun}

\thanks{Supported in part by grants from NSFC (No. 11371093, and 11431009).}
\address{School of Mathematical Sciences, Fudan University, Shanghai,
200433, P.R. China. } \email{zlu@fudan.edu.cn}
\address{School of Mathematical Sciences, Fudan University, Shanghai,
200433, P.R. China. Current address: Department of Mathematics, Capital Normal University, Beijing, 100048, P. R. China. }
\email{tonglunlun@gmail.com}
\address{School of Mathematical Sciences, Fudan University, Shanghai,
200433, P.R. China.}
\email{09210180013@fudan.edu.cn}

\date{}
\begin{abstract}
Let $P$ be a simple polytope of dimension $n$ with $m$ facets. In this paper we  pay our attention on those elementary symmetric polynomials in the
Stanley--Reisner face ring of $P$ and  study how the decomposability of the $n$-th elementary symmetric polynomial influences on the combinatorics of $P$ and the topology and geometry of toric spaces over $P$. We give  algebraic criterions of detecting the decomposability of $P$ and determining when $P$ is $n$-colorable in terms of the $n$-th elementary symmetric polynomial. In addition, we define the Stanley--Reisner {\em exterior} face ring $\mathcal{E}(K_P)$ of $P$, which is non-commutative in the case of ${\Bbb Z}$ coefficients, where $K_P$ is the boundary complex of dual of $P$. Then we obtain a criterion
for the (real) Buchstaber invariant of $P$ to be $m-n$ in terms of the $n$-th elementary symmetric polynomial in $\mathcal{E}(K_P)$.
Our results as above can directly associate with the topology and geometry of toric spaces over $P$. In particular, we show that  the decomposability of the $n$-th elementary symmetric polynomial in $\mathcal{E}(K_P)$ with ${\Bbb Z}$ coefficients  can  detect the existence of the almost complex structures of quasitoric manifolds over $P$, and if the (real) Buchstaber invariant of $P$ is $m-n$, then  there exists an essential relation between the $n$-th equivariant characteristic class  of the (real) moment-angle manifold over $P$ in $\mathcal{E}(K_P)$ and the characteristic functions of $P$.
\end{abstract}
\maketitle


\section{Introduction}

The Stanley--Reisner face ring is  a fundamental  tool in algebraic combinatorics and combinatorial commutative algebra (\cite{ms}, \cite{st}). In addition, it is well-known that the Stanley--Reisner face ring can also be realizable as the equivariant cohomology of many toric spaces over a simple polytope $P$, such as toric varieties, (quesi-)toric manifolds, small covers and (real) moment-angle manifolds.
Moreover, elementary symmetric polynomials in the Stanley--Reisner face ring can be realized as the  equivariant Chern classes or equivariant Stiefel--Whitney classes of those toric spaces. Thus, the Stanley--Reisner face ring
also plays an essential role on algebraic geometry, toric geometry and toric topology (\cite{bp1}).

\vskip .1cm

Let $K$ be an abstract simplicial complex of dimension $n-1$ on vertex set $[m]=\{1, ..., m\}$, and let $R$ be a commutative ring with unit. Then the  {\em Stanley--Reisner face ring} $R(K)$ of $K$ is defined as the quotient ring of
 the polynomial ring $R[x_1, ..., x_m]$
$$R(K)=R[x_1, ..., x_m]/\mathcal{I}_K$$
where $\mathcal{I}_P$ is the ideal generated by those square-free monomials $x_{i_1}\cdots x_{i_r}$ for which $\{i_1, ..., i_r\}\not\in K$.
Let
$$\sigma_i^K(x_1, ..., x_m)=\sum x_{j_1}\cdots x_{j_i}$$
be the {\em $i$-th elementary symmetric polynomial function} in Stanley--Reisner face ring $R(K)$, which is the image
 of
the standard $i$-th elementary symmetric polynomial function $\sigma_i(x_1, ..., x_m)$ in $R[x_1, ..., x_m]$ under the natural projection
$$p: R[x_1, ..., x_m]\longrightarrow R(K)=R[x_1, ..., x_m]/\mathcal{I}_K.$$
In general, $\sigma_i^K(x_1, ..., x_m)$ is a  deficient function in $R(K)$ since  some monomials of $\sigma_i(x_1, ..., x_m)$ in $R[x_1, ..., x_m]$ may be missing in $R(K)$. However,
 these $\sigma_i^K(x_1, ..., x_m)$ record the complete combinatorics of $K$. In fact, for two abstract simplicial complexes $K_1, K_2$ of dimension $n-1$ on $[m]$,  we can regard $R(K_1)$ and $R(K_2)$ as two quotient rings of $R[x_1, ..., x_m]$. Then it is easy to see that $K_1$ and  $K_2$ are combinatorially isomorphic if and only if there is a permutation $s$ on $\{x_1, ..., x_m\}$ such that
 $\sigma_i^{K_1}(x_1, ..., x_m)=\sigma_i^{K_2}(s(x_1), ..., s(x_m))$ for all $i$.

 \vskip .1cm

In this paper we shall investigate  the internal implications (e.g., decomposability) of those polynomials  $\sigma_i^{K}(x_1, ..., x_m)$ 
 in some different rings,  such as $R(K), R[x_1, ..., x_m]$ and the  Stanley--Reisner exterior face ring etc.  Our main purpose  is to study how the internal implications  of those polynomials  $\sigma_i^{K}(x_1, ..., x_m)$
  can produce essential influences on the combinatorics of $K$ and the topology and geometry of toric spaces.
 Here we will pay more attention on the case in which $K$ is the boundary complex $K_P$ of the dual polytope of a simple $n$-polytope $P$ with $m$ facets although our many arguments can still be carried on for the case of more general $K$. This is because in this case we can directly associate with the topology and geometry of toric spaces over $P$, and in particular, $\sigma_n^{K_P}(x_1, ..., x_m)$ completely determines the combinatorics of $K_P$.

\vskip .1cm
Given an  $n$-dimensional simple polytope $P$ with $m$ facets,
  we shall carry out our work in the following some aspects:
 \begin{enumerate}
 \item  We show that the decomposability of $P$ agrees with that of $\sigma_n^{K_P}(x_1, ..., x_m)$ in $R[x_1, ..., x_m]$ (see Theorem~\ref{main1}), where the decomposability of $P$ means whether $P$ is a product of some polytopes or not,
  and  we note that there is a natural module embedding $R(K_P)\hookrightarrow R[x_1, ..., x_m]$, so $\sigma_n^{K_P}(x_1, ..., x_m)$ in $R(K_P)$ can be regarded as a polynomial in $R[x_1, ..., x_m]$.
         This gives an algebraic criterion of detecting the decomposability of $P$ in terms of $\sigma_n^{K_P}(x_1, ..., x_m)$ in $R[x_1, ..., x_m]$. It should be pointed out that the decomposability of $\sigma_n^{K_P}(x_1, ..., x_m)$ in $R(K_P)$ is different from that of $\sigma_n^{K_P}(x_1, ..., x_m)$ in $R[x_1, ..., x_m]$. In other words, we cannot use the decomposability of $\sigma_n^{K_P}(x_1, ..., x_m)$ in $R(K_P)$
     to detect that of $P$. This work is motivated by \cite[Section 7, Problem (P2)]{lt}. Here we give an answer to  \cite[Section 7, Problem (P2)]{lt} in spite of the existence of  characteristic functions on $P$.
 \item With a combinatorial argument, we give a criterion  for $P$ to be $n$-colorable in terms of the decomposability of $\sigma_n^{K_P}(x_1, ..., x_m)$ in $R(K_P)$ (see Theorem~\ref{n-co}). More generally, for each integer $\ell\geq n$,  we can also give a  criterion for $P$ to be $\ell$-colorable in terms of $\sigma_n^{K_P}(x_1, ..., x_m)$ in $R(K_P)$ (see Theorem~\ref{l-co}).
     With a topological argument, Notbohm in~\cite{n1, n2} has given a criterion for $P$ to be $\ell$-colorable in terms of the splitting property or total characteristic class of a vector bundle.
     However, our criterion is a little bit weaker than Notbohm's one. In fact,
      Notbohm's criterion needs to involve all elementary symmetric polynomials   $\sigma_i^{K_P}(x_1, ..., x_m), i\leq n$,
      but  our criterion only needs the $n$-th elementary symmetric polynomial $\sigma_n^{K_P}(x_1, ..., x_m)$.
     \item We study the Buchstaber invariant $s_{\C}(P)$ (or $s_{\R}(P)$) of $P$, which is a combinatorial invarant, introduced by Buchstaber, and has also the geometrical meaning (\cite{bp}). Actually,  $s_{\C}(P)$ (or $s_{\R}(P)$) is related to the existence of the free actions of the maximal subtori on moment-angle manifold $\mathcal{Z}^{\C}_P$ (or real moment-angle manifold $\mathcal{Z}^{\R}_P$) over $P$.
      In a similar way to the Stanley--Reisner face ring, we introduce the  Stanley--Reisner {\em exterior} face ring
     $\mathcal{E}_{\Bbb F}(K_P)$ over $R_{\Bbb F}$ of $P$, where  ${\Bbb F}=\C$ or $\R$, and $R_{\Bbb F}$ is $\Z$ if ${\Bbb F}=\C$ and $\Z_2$ if ${\Bbb F}=\R$. This ring $\mathcal{E}_{\Bbb F}(K_P)$ is non-commutative if ${\Bbb F}=\C$.
     Then we give a necessary and sufficient condition for $s_{\Bbb F}(P)=m-n$ in terms of the decomposability of $\sigma_n^{K_P}(x_1, ..., x_m)$ in $\mathcal{E}_{\Bbb F}(K_P)$ (see Theorem~\ref{cpx} and Theorem~\ref{real}), where if ${\Bbb F}=\C$,  $\sigma_n^{K_P}(x_1, ..., x_m)$ will be an oriented polynomial associated with an orientation of $K_P$.
\end{enumerate}

Our work as above can directly associate with the topology and the geometry of toric spaces over $P$  although most of our results and proofs are combinatorial and algebraic, so that we can obtain more understandings from the viewpoints of topology and geometry. In fact, $\sigma_n^{K_P}(x_1, ..., x_m)$ in $R_{\Bbb F}(K_P)$
is exactly the $n$-th equivariant Chern class (or equivariant Stiefel--Whitney class) of $\mathcal{Z}^{\C}_P$ (or $\mathcal{Z}^{\R}_P$). Thus, $\sigma_n^{K_P}(x_1, ..., x_m)$ in $R_{\Bbb F}[x_1, ..., x_m]$ is the pullback of the
$n$-th equivariant characteristic class $\sigma_n^{K_P}(x_1, ..., x_m)$ in $R_{\Bbb F}(K_P)$ via the embedding $R_{\Bbb F}(K)\hookrightarrow R_{\Bbb F}[x_1, ..., x_m]$. We will also see that if ${\Bbb F}=\C$, the decomposability of $\sigma_n^{K_P}(x_1, ..., x_m)$ can  detect the existence of the almost complex structures of quasitoric manifolds over $P$ (see Proposition~\ref{al-com}), and in particular, we give a simple criterion for determining whether a quasitoric manifold admits an equivariant almost complex structure (see Corollary~\ref{cycle}). In addition, our necessary and sufficient condition for $s_{\Bbb F}(P)=m-n$ also implies that there exists an essential relation between the $n$-th equivariant characteristic class $\sigma_n^{K_P}(x_1, ..., x_m)$ of $\mathcal{Z}^{\C}_P$ (or  $\mathcal{Z}^{\R}_P$) in $\mathcal{E}_{\Bbb F}(K_P)$ and the characteristic functions of $P$ (see Corollary~\ref{ch-pol1} and Corollary~\ref{ch-pol2}). An application to the Buchstaber invariants of cyclic polytopes is also given.


\section{Geometric realization of the Stanley--Reisner face ring in toric topology}

\subsection{Moment-angle manifolds}
An $n$-dimensional polytope $P$ is said to be {\em simple} if each vertex of $P$ meets exactly $n$ facets.  Simple polytopes have played essential roles in the theory of  toric varieties, toric geometry and toric topology.
 Let $P$ be an $n$-dimensional simple polytope with $m$ facets $F_1, ..., F_m$. Then its dual $P^*$ is a simplicial polytope of dimension $n$, and the boundary $\partial P^*$ of $P^*$, denoted by $K_P$,  is a simplicial complex of dimension $n-1$
 on vertex set $[m]$ (corresponding to the facet set $\{F_1, ..., F_m\}$ of $P$). The complex $K_P=\partial P^*$ is also called the {\em polytopal sphere} of $P$. It is well-known that over $P$,
the moment-angle manifold $\mathcal{Z}^{\C}_P$ and  real moment-angle manifold $\mathcal{Z}^{\R}_P$ can be defined with different ways.
Following \cite[Construction 6.38]{bp}, for each simplex $\sigma$ in
$K_P$, set
$B_\sigma=\prod_{i=1}^m A_i$
such that $$A_i=\begin{cases} D^2 & \text{if } i\in \sigma\\
S^1 & \text{if } i\in [m]\setminus\sigma
\end{cases}$$
where  $D^2=\big\{z\in {\Bbb C}\big| |z|\leq 1\big\}$ is the unit disk in ${\Bbb
 C}$, and $S^1=\partial D^2$. Then one can define the moment-angle manifold $\mathcal{Z}^{\C}_P$ as the following subspace of the product
space $(D^2)^m$:
$$\mathcal{Z}^{\Bbb C}_P:=\bigcup_{\sigma\in K_P}B_\sigma\subset (D^2)^m.$$
The manifold $\mathcal{Z}^{\Bbb C}_P$ admits a natural action of $\mathbb{T}_{\Bbb C}^m=(S^1)^m$, which is the restriction to $\mathcal{Z}^{\Bbb C}_P$ of the standard $\mathbb{T}_{\Bbb C}^m$-representation on ${\Bbb C}^m$
given by $$\big((g_1,...,
g_m), (z_1, ..., z_m)\big)\longmapsto (g_1z_1, ..., g_mz_m).$$ The complex conjugation on ${\Bbb C}^m$ gives a conjugation involution on $\mathcal{Z}^{\Bbb C}_P$ whose fixed point set is exactly the real moment-angle manifold $\mathcal{Z}^{\Bbb R}_P\subset (D^1)^m\subset {\Bbb R}^m$ and admits an action of the elementary 2-group $\mathbb{T}_{\Bbb R}^m=(S^0)^m\cong {\Bbb Z}_2^m$, where $\mathbb{T}_{\Bbb R}^m$ is the fixed point set of the complex conjugation on $\mathbb{T}_{\Bbb C}^m\subset {\Bbb C}^m$. As shown in~\cite{dj}, $\mathcal{Z}^{\Bbb F}_P$ plays an important role on the study of quasitoric manifolds and small covers where ${\Bbb F}={\Bbb C}$ or ${\Bbb R}$.
We note that each $\mathcal{Z}^{\Bbb R}_P$ itself can be defined in a similar way to $\mathcal{Z}^{\Bbb C}_P$, but it can always be realized as the fixed point set of the conjugation involution on $\mathcal{Z}^{\Bbb C}_P$. However, generally both $\mathcal{Z}^{\Bbb C}_P$ and $\mathcal{Z}^{\Bbb R}_P$  have  quite differences, as shown in~\cite[Lemma 6.5]{dj}.
One also
knows from  \cite{dj} that both $\mathbb{T}_{\Bbb C}^m$-action  on
$\mathcal{Z}^{\Bbb C}_P$ and $\mathbb{T}_{\Bbb R}^m$-action  on
$\mathcal{Z}^{\Bbb R}_P$ have the same orbit space  $P$.

\subsection{Equivariant cohomology and equivariant characteristic classes of $\mathcal{Z}^{\Bbb F}_P$}
Davis and Januszkiewicz showed  in~\cite{dj} that the cohomology of the Borel construction $E\mathbb{T}_{\Bbb F}^m\times_{\mathbb{T}_{\Bbb F}^m}\mathcal{Z}^{\Bbb F}_P$ of
$\mathcal{Z}^{\Bbb F}_P$ only depends upon $P$, and it is isomorphic to the Stanley--Reisner face ring of $P$.
\begin{theorem}  [{\cite[Theorem 4.8]{dj}}] \label{face ring}  There is a ring isomorphism
$$H^*_{\mathbb{T}_{\Bbb F}^m}(\mathcal{Z}^{\Bbb F}_P; R_{\Bbb F})\cong R_{\Bbb F}(K_P)=R_{\Bbb F}[x_1, ..., x_m]/\mathcal{I}_{K_P}$$
with $\deg x_i=\dim {\Bbb F}$,
where $R_{\Bbb F}$ is $\Z$ if ${\Bbb F}=\C$ and $\Z_2$ if ${\Bbb F}=\R$.
\end{theorem}
In~\cite{dj}, Davis and Januszkiewicz  constructed $m$ canonical line bundles $\mathbb{L}_i (i=1, ..., m)$ over $E\mathbb{T}_{\Bbb F}^m\times_{\mathbb{T}_{\Bbb F}^m}\mathcal{Z}^{\Bbb F}_P$, which are stated as follows. First, let
$\rho_i: \mathbb{T}_{\Bbb F}^m\longrightarrow \mathbb{T}_{\Bbb F}$ be the projection onto the $i$-th factor and let ${\Bbb F}(\rho_i)$ denote the corresponding
1-dimensional $\mathbb{T}_{\Bbb F}^m$-representation space. Then, define a trivial equivariant line bundle $\widetilde{\mathbb{L}}_i$ over $\mathcal{Z}^{\Bbb F}_P$ by $\widetilde{\mathbb{L}}_i={\Bbb F}(\rho_i)\times \mathcal{Z}^{\Bbb F}_P$. Finally, consider the Borel construction on $\widetilde{\mathbb{L}}_i$, we
may obtain the required line bundle
$$\mathbb{L}_i=E\mathbb{T}_{\Bbb F}^m\times_{\mathbb{T}_{\Bbb F}^m}\widetilde{\mathbb{L}}_i$$
and the first Stiefel--Whitney class $w_1(\mathbb{L}_i)=x_i$ if ${\Bbb F}=\R$, and first Chern class $c_1(\mathbb{L}_i)=x_i$ if ${\Bbb F}=\C$.
 In fact, $\mathbb{L}_i$ exactly corresponds to the facet $F_i$ of $P$. Furthermore,  Davis and Januszkiewicz showed in~\cite[Theorem 6.6]{dj} that the Borel construction $E\mathbb{T}^m_{\Bbb F}\times_{\mathbb{T}^m_{\Bbb F}}\mathcal{T}\mathcal{Z}_P^{\Bbb F}$ of the tangent bundle $\mathcal{T}\mathcal{Z}_P^{\Bbb F}$
is stably isomorphic to the Whitney sum $\mathbb{L}_1\oplus\cdots\oplus \mathbb{L}_m$ as real vector bundles.  Thus, if ${\Bbb F}=\R$,
the total equivariant Stiefel--Whitney class of $\mathcal{T}\mathcal{Z}_P^{\Bbb R}$ is
$$w^{\mathbb{T}^m_{\Bbb R}}(\mathcal{T}\mathcal{Z}_P^{\Bbb R})=\prod_{i=1}^m(1+x_i) \text{ in } H^*_{\mathbb{T}^m_{\Bbb R}}(\mathcal{Z}_P^{\Bbb R}; {\Bbb Z}_2)\cong R_{\Bbb R}[x_1, ..., x_m]/\mathcal{I}_{K_P}$$
(see also~\cite[Corollary 6.7]{dj}). If ${\Bbb F}=\C$, following the arguments of \cite[Theorem 7.3.15]{bp1},
there is the following $\mathbb{T}^m_{\C}$-equivariant decomposition by restricting the tangent bundle $\mathcal{T}\C^m$ to $\mathcal{Z}_P^{\C}$
\begin{equation}\label{bundle-iso}
\mathcal{T}\mathcal{Z}_P^{\Bbb C}\oplus \nu(i_{\mathcal{Z}})\cong \mathcal{Z}_P^{\C}\times \C^m
\end{equation}
where $i_{\mathcal{Z}}: \mathcal{Z}_P^{\C}\hookrightarrow \C^m$ is a $\mathbb{T}^m_{\C}$-equivariant embedding, and $\nu(i_{\mathcal{Z}})$ is the normal bundle of $\mathcal{Z}_P^{\C}$ in $\C^m$ which is $\mathbb{T}^m_{\C}$-equivariantly trivial. Since $\mathcal{Z}_P^{\C}\times \C^m$ is isomorphic to $\widetilde{\mathbb{L}}_1\oplus\cdots\oplus \widetilde{\mathbb{L}}_m$ as $\mathbb{T}^m_{\C}$-equivariant vector bundles,  applying the Borel construction to the bundle isomorphism~(\ref{bundle-iso}), one has that
$$\big(E\mathbb{T}_{\Bbb C}^m\times_{\mathbb{T}_{\Bbb C}^m}\mathcal{T}\mathcal{Z}_P^{\Bbb C}\big)\oplus
\big(E\mathbb{T}_{\Bbb C}^m\times_{\mathbb{T}_{\Bbb C}^m} \nu(i_{\mathcal{Z}})\big)\cong \mathcal{Z}_P^{\C}\times_{\mathbb{T}_{\C}^m} \C^m=\mathbb{L}_1\oplus\cdots\oplus \mathbb{L}_m$$
so the total equivariant Chern class of $\mathcal{T}\mathcal{Z}_P^{\Bbb C}$ is
$$c^{\mathbb{T}^m_{\Bbb C}}(\mathcal{T}\mathcal{Z}_P^{\Bbb C})=\prod_{i=1}^m(1+x_i) \text{ in } H^*_{\mathbb{T}^m_{\Bbb C}}(\mathcal{Z}_P^{\Bbb C}; {\Bbb Z})\cong R_{\Bbb C}[x_1, ..., x_m]/\mathcal{I}_{K_P}$$
and thus $w_i^{\mathbb{T}^m_{\Bbb R}}(\mathcal{T}\mathcal{Z}_P^{\Bbb R})$ and $c_i^{\mathbb{T}^m_{\Bbb C}}(\mathcal{T}\mathcal{Z}_P^{\Bbb C})$
can be written as $\sigma_i^{K_P}(x_1, ..., x_m)$.

\vskip .1cm

We note that since $H^*(B\mathbb{T}_{\Bbb F}^m; R_{\Bbb F})=R_{\Bbb F}[x_1, ..., x_m]$,  the natural projection $$p: R_{\Bbb F}[x_1, ..., x_m]\longrightarrow R_{\Bbb F}(K_P)=R_{\Bbb F}[x_1, ..., x_m]/\mathcal{I}_{K_P}$$ is actually induced by the fiberation $\pi: E\mathbb{T}_{\Bbb F}^m\times_{\mathbb{T}_{\Bbb F}^m}\mathcal{Z}^{\Bbb F}_P\longrightarrow B\mathbb{T}_{\Bbb F}^m$, so $\sigma_i^{K_P}(x_1, ..., x_m)$ is the image of
the $i$-th universal characteristic class $\sigma_i(x_1, ..., x_m)$ in $R_{\Bbb F}[x_1, ..., x_m]$ under the homomorphism $\pi^*$.
Generally, each polynomial $\overline{f}$ in $R_{\Bbb F}(K_P)$ is a coset $f+\mathcal{I}_{K_P}$ where $f\in R_{\Bbb F}[x_1, ..., x_m]$. Since
the ideal $\mathcal{I}_{K_P}$ is exactly generated by square free monomials, each coset $f+\mathcal{I}_{K_P}$ contains a unique representative that has no any monomial in $\mathcal{I}_{K_P}$. Such representative will be said to be {\em prime}. Then we have
\begin{lemma}
There is a module embedding
$e: R_{\Bbb F}(K_P)\hookrightarrow R_{\Bbb F}[x_1, ..., x_m]$ defined by mapping $\overline{f}$ to its prime representative.
\end{lemma}

Without any confusion, we will identify $\overline{f}$ with its prime representative. With this understanding,  $R_{\Bbb F}(K_P)$ can be regarded as a submodule of $R_{\Bbb F}[x_1, ..., x_m]$, and in particular,    $\sigma_i^{K_P}(x_1, ..., x_m)$ in $R_{\Bbb F}(K_P)$ can also understood as a polynomial in $R_{\Bbb F}[x_1, ..., x_m]$.

 \subsection{Buchstaber invariant}
 In general, the action of $\mathbb{T}_{\Bbb F}^m$ on $\mathcal{Z}^{\Bbb F}_P$ is not free, but the action restricted to some sub-tori of $\mathbb{T}_{\Bbb F}^m$ may be free.
There is the maximum rank of those sub-tori that can freely on $\mathcal{Z}^{\Bbb F}_P$, denoted by $s_{\Bbb F}(P)$, which is called {\em Buchstaber invariant} in~\cite{bp}. It was known from~\cite{bp} that
\begin{proposition}
$1\leq s_{\Bbb C}(P)\leq s_{\Bbb R}(P)\leq m-n$.
\end{proposition}

If $\dim P=3$, by Four Color Theorem, then $s_{\Bbb C}(P)= s_{\Bbb R}(P)= m-3$. Generally, $s_{\Bbb C}(P)$ and  $s_{\Bbb R}(P)$ may be strictly less than $m-n$. This can be seen from cyclic polytopes of dimension $\geq 4$.
Buchstaber invariant can also be defined for the general simplicial complexes, and it is a combinatorial invariant. However, the calculation of Buchstaber invariant is quite difficult and complicated (\cite[Problem 7.27]{bp}). Some works on the properties of Buchstaber invariant and the calculations in some special cases have been carried on (see, e.g.~\cite{a1, a2, e1, e2, e3, fm, sy}).

\subsection{Quasitoric manifolds and small covers}\label{char}
In their seminar work~\cite{dj}, Davis and Januszkiewicz introduced and studied quasitoric manifolds and small covers. Let $P$ be an $n$-dimensional simple polytope with $m$ facets $F_1, ..., F_m$.
If Buchstaber invariant $s_{\Bbb F}(P)=m-n$, then we can choose a subtorus $H_{\Bbb F}$ of rank $m-n$ in $\mathbb{T}_{\Bbb F}^m$ which acts freely on $\mathcal{Z}_P^{\Bbb F}$, so that
the quotient space $\mathcal{Z}_P^{\Bbb F}/H_{\Bbb F}$ is a closed manifold and admits an action of $\mathbb{T}_{\Bbb F}^m/H_{\Bbb F}\cong \mathbb{T}_{\Bbb F}^n$. The space $\mathcal{Z}_P^{\Bbb F}/H_{\Bbb F}$
is called a {\em quasitoric manifold} if ${\Bbb F}=\C$, and a {\em small cover} if ${\Bbb F}=\R$. Conversely, if there is a quasitoric manifold (or small cover) over $P$, then it is easy to see that $s_{\Bbb C}(P)=m-n$
(or $s_{\Bbb R}(P)=m-n$). Thus,
\begin{proposition}
There exist quasirotic manifolds (resp. small covers) over $P$ if and only if $s_{\Bbb C}(P)=m-n$ (resp. $s_{\Bbb R}(P)=m-n$).
\end{proposition}

Let $\pi_{\Bbb F}: M_{\Bbb F}\longrightarrow P$ be a quasitoric manifold or a small cover over $P$, where $\dim M_{\Bbb F}=\dim {\Bbb F}\cdot n$ and $ M_{\Bbb F}$ admits an action of $\mathbb{T}_{\Bbb F}^n$. Davis and Januszkiewicz showed  in~\cite{dj} that the Stanley--Reisner face ring can also be realizable as the equivariant cohomology of $M_{\Bbb F}$. 
\begin{theorem}[{\cite[Theorem 4.8]{dj}}]
There is a ring isomorphism
$$H^*_{\mathbb{T}_{\Bbb F}^n}(M_{\Bbb F};R_{\Bbb F})\cong R_{\Bbb F}(K_P)=R_{\Bbb F}[x_1, ..., x_m]/\mathcal{I}_{K_P}$$
with $\deg x_i=\dim {\Bbb F}$,
where $R_{\Bbb F}$ is $\Z$ if ${\Bbb F}=\C$ and $\Z_2$ if ${\Bbb F}=\R$.
 \end{theorem}
 In addition, as shown in~\cite{dj}, the action of $\mathbb{T}_{\Bbb F}^n$ on $M_{\Bbb F}$ gives an essential information on $P$ via $\pi_{\Bbb F}$, which is just the {\em characteristic function}
$$\lambda: \mathcal{F}(P)=\{F_1, ..., F_m\}\longrightarrow R_{\Bbb F}^n$$
such that for each vertex $v$ of $P$ (so there are exactly $n$ facets, say $F_{i_1}, ..., F_{i_n}$, such that $v=F_{i_1}\cap\cdots\cap F_{i_n}$), $\lambda(F_{i_1}), ..., \lambda(F_{i_n})$ form a basis of
$R_{\Bbb F}^n$.  The characteristic function
$\lambda$ can naturally be spanned into a linear map $R_{\Bbb F}^m\longrightarrow R_{\Bbb F}^n$, also denoted by $\lambda$. Fixing an ordering of all facets of $P$, say $F_1, ..., F_m$, then $\lambda$ determines a unique
$(n\times m)$-matrix $\Lambda=(\lambda_{ij})$. As shown in \cite{dj}, the transpose of this matrix $\Lambda=(\lambda_{ij})$ is actually identified with the map $p^*: H^{\dim {\Bbb F}}(B\mathbb{T}_{\Bbb F}^n;R_{\Bbb F})\longrightarrow
H^{\dim {\Bbb F}}_{\mathbb{T}_{\Bbb F}^n}(M_{\Bbb F};R_{\Bbb F})$, where $p^*$ is induced by the fibration $p: E\mathbb{T}_{\Bbb F}^n\times_{\mathbb{T}^n_{\Bbb F}}M_{\Bbb F}\longrightarrow B\mathbb{T}^n_{\Bbb F}$. Thus, the characteristic function
$\lambda$ can be regarded as a sequence $$(\lambda_1, ..., \lambda_n)$$ (still denoted by $\lambda$) in $H^*_{\mathbb{T}_{\Bbb F}^n}(M_{\Bbb F};R_{\Bbb F})\cong R_{\Bbb F}[x_1, ..., x_m]/\mathcal{I}_{K_P}$,
where $\lambda_i=\lambda_{i1}x_1+\cdots+\lambda_{im}x_m$.  The sequence $\lambda=(\lambda_1, ..., \lambda_n)$ determines an ideal $J_\lambda=\langle\lambda_1, ..., \lambda_n\rangle$ of $R_{\Bbb F}(K_P)$. Then
\begin{theorem} [{\cite[Theorem 4.14]{dj}}]
The cohomology of $M_{\Bbb F}$ is the quotient ring of $R_{\Bbb F}(K_P)$ by $J_\lambda$. Namely
$$H^*(M_{\Bbb F};R_{\Bbb F})\cong R_{\Bbb F}(K_P)/J_\lambda.$$
\end{theorem}

\begin{remark}
As shown in~\cite{bpr}, the stably complex structure on a quasitoric manifold $M_{\Bbb C}$ depends upon the choice of omniorientations. For more details, see~\cite{bpr}.
\end{remark}

\section{Decomposability of simple polytopes}

\subsection{Abstract simplicial complexes and their joins}
An {\em abstract simplicial complex} $K$ on a finite set $S$ is a collection in the power set $2^S$ such that
 for each $a\in K$, any subset (including empty set) of $a$ still belongs to $K$. Each $a$ in $K$ is called a simplex and has dimension $|a|-1$, where
 $|a|$ is the cardinality of $a$.  The dimension of $K$ is defined as $\max\limits_{a\in K}\{\dim a\}$.

\vskip .1cm
 It is well-known (cf. \cite{bp}) that up to combinatorial equivalence,  each finite abstract simplicial complex can be realized as a unique geometric simplicial complex; and conversely, each finite geometric simplicial complex gives a unique abstract simplicial complex. Thus, we may identify finite abstract simplicial complexes with finite geometric simplicial complexes.
 Throughout the following,  a simplicial complex will mean a finite abstract simplicial complex or a finite geometric simplicial complex.
\vskip .1cm

 Let $K_1$ and $K_2$ be two  simplicial complexes on finite sets $S_1$ and $S_2$, respectively. Then the {\em join} of $K_1$ and $K_2$ is an abstract simplicial
 complex
 $$K_1\ast K_2=\{a\cup b\in 2^{S_1\cup S_2}|a\in K_1, b\in K_2\}$$
 on the set $S_1\cup S_2$.

  \subsection{Polynomials of simplicial complexes} Let $K$ be a simplicial complex on the set $S$ with $|S|=m$. Without the loss of generality, assume that $K$ is a simplicial complex on $[m]=\{1, ..., m\}$.
 Now consider the commutative polynomial ring $R[x_1, ..., x_m]$ where $R$ is a commutative ring with unit and  the $x_i$ are indeterminants of same degree. Regarded $K$ as a poset with respect to the inclusion, for each maximal element $a=\{i_1, ..., i_r\}$ in $K$,
 we define a monomial $x_{i_1}\cdots x_{i_r}$, denoted by $m_a$. Then we obtain a square-free polynomial
 $$f_K=\sum_am_a$$
 where $a$ runs over all maximal elements in $K$. We call $f_K$ the {\em polynomial of $K$}.
\vskip .1cm
 If $K$ is the boundary complex $K_P$ of the dual polytope of an $n$-dimensional simple polytope $P$ with $m$ facets,
 then $f_{K_P}$ is a square-free homogeneous polynomial in $R[x_1, ..., x_m]$, and it is exactly the $n$-th elementary symmetric function $\sigma_n^{K_P}(x_1, ..., x_m)$ in $R_{\Bbb F}(K_P)\hookrightarrow R_{\Bbb F}[x_1, ..., x_m]$.

\vskip .1cm

 A polynomial $f$ in $R[x_1, ..., x_m]$ is said to be {\em nice} if the coefficients of all monomials of $f$ are 1 and each monomial of $f$ is not a factor of other monomials of $f$. Let $f=\sum x_{i_1}\cdots x_{i_r}$ be a nice polynomial in  $R[x_1, ..., x_m]$. Clearly each monomial $x_{i_1}\cdots x_{i_r}$ determines a subset $\{i_1, ..., i_r\}$ of $[m]$, called the {\em monomial subset} of $x_{i_1}\cdots x_{i_r}$. Then we obtain a simplicial complex $K_f$ with all monomial subsets as its maximal elements. This gives

 \begin{lemma}
All simplicial complexes on $[m]$ bijectively correspond to all nice polynomials in $R[x_1, ..., x_m]$.
 \end{lemma}

 \begin{proposition}\label{s-poly}
 Let $K$ be a simplicial complex on $[m]$. Then $K$ is a join of two simplicial complexes if and only if $f_K$ is a product of two  polynomials in $R[x_1, ..., x_m]$.
 \end{proposition}

 \begin{proof}
 If $K=K_1\ast K_2$, then it is easy to see that $f_K=f_{K_1}f_{K_2}$. Conversely, if $f=f_1f_2$, since $f$ is square free, both $f_1$ and $f_2$ are square free, too. Moreover, both $f_1$ and $f_2$ are nice, and $K_{f_1}$, $K_{f_2}$ share no common vertices.
  Then we see easily that $K_f=K_{f_1}\ast K_{f_2}$.
 \end{proof}

 \subsection{The decomposability of $P$} Now let us discuss the  decomposability of a simple polytope $P$.

 \begin{lemma}[\cite{bp}]\label{formula}
Let $P_1$ and $P_2$ be  simple polytopes.  Then
$$K_{P_1}*K_{P_2}=K_{P_1\times P_2}.$$
\end{lemma}

\begin{lemma}\label{N}
If $K$ is a polytopal sphere, then the link of any vertex of $K$ is also a polytopal sphere.
\end{lemma}
\begin{proof}
Assume $K=\partial P^*$ where $P$ is a simple polytope with the facet set $\mathcal{F}(P)=\{F_1,F_2,\dots F_m\}$. We note that $K$ may be regarded as the simplicial complex with $\mathcal{F}(P)$ as its vertex set such that
each simple $\sigma$ of $K$ is a subset $\{F_{i_1}, ..., F_{i_r}\}$ with $F_{i_1}\cap\cdots\cap F_{i_r}\not=\varnothing$ of $\mathcal{F}(P)$.
With this understood, now given a vertex  $v$ in $K$, then there is a facet $F$ of $P$ such that $v=\{F\}$. Thus
\begin{eqnarray*}
 \begin{split}
\textrm{Link}_K(v)&=\big\{\tau=\{F_{i_1}, ..., F_{i_r}\}\in K~| v\cup\tau\in K, v\cap\tau=\varnothing\big\}\\
&=\{\{F_{i_1},\ldots,F_{i_r}\}\in K~|~F\cap F_{i_1}\cap\cdots\cap F_{i_r}\neq\varnothing\}.\\
 \end{split}
\end{eqnarray*}
Clearly, $\textrm{Link}_K(v)$
 is combinatorially equivalent to the simplicial complex
$K'=\{\{F\cap F_{i_1},\ldots,F\cap F_{i_r}\}~|~(F\cap F_{i_1})\cap\cdots\cap (F\cap F_{i_r})\neq\varnothing,~F\cap F_{i_r}\in\mathcal{F}(F)\}$ with $\mathcal{F}(F)$ as vertex set, where $\mathcal{F}(F)$ denotes the set of all facets of $F$.
Since $P$ is simple, $F$ is also simple, so $K'$ is the boundary complex $\partial F^*$ of the dual of $F$. Therefore, $\textrm{Link}_K(v)$ is a polytopal sphere.
\end{proof}

\begin{lemma}\label{h}
Let $K'$ be a simplical complex
and $K$ be the suspension of $K'$ such that $K=K'*S^0$, where $S^0$ is the $0$-dimensional sphere. Then $K$ is a polytopal sphere if and only if $K'$ is a polytopal sphere.
\end{lemma}
\begin{proof}
We note that $S^0$ is the polytopal sphere of the interval $I$ as a 1-dimensional  polytope. So $S^0=\partial I^*$. Let $v_1$ and $v_2$ be two vertices of $I$. Then $S^0$ is the 0-dimensional  simplicial complex on vertex set $\{v_1, v_2\}$. Furthermore, $K=K'*S^0=\big\{\tau, \tau\cup\{v_1\}, \tau\cup\{v_2\}~|~\tau\in K'\big\}$, i.e., the suspension of $K'$.
\vskip .1cm

If $K$ is a polytopal sphere,  obviously $\textrm{Link}_K(v_1)=\textrm{Link}_K(v_2)=K'$. By Lemma \ref{N}, $K'$ is a polytopal sphere. Conversely, if $K'$ is a polytopal sphere, we may assume that $K'=\partial P'^*$ for some simple polytope $P'$. By Lemma~\ref{formula}, we have $K=K'*S^0=(\partial P'^*)*(\partial I^*)=\partial(P'\times I)^*.$
So $K$ is exactly the polytopal sphere of $P'\times I$.
\end{proof}

\begin{lemma}\label{s}
Let $K, K_1$ and $K_2$ be simplicial complexes such that $K=K_1*K_2$. If $K$ is a polytopal sphere, then both $K_1$ and $K_2$ are  polytopal spheres, too.
\end{lemma}

\begin{proof}
We will perform an induction on the dimension of $K$. We note that $\dim K=\dim K_1+\dim K_2+1\geq 1$.

When $\dim K=1$, since $K$ is a polytopal sphere, it must be a circle, so it would be a boundary complex of the dual of a  polygon. It is not hard to check that $K_1=K_2=S^0$. So the theorem follows.

Now assume inductively that   the theorem holds if $\dim K \leq n$.
When $\dim K=n+1$, by Lemma \ref{h} we know that $K*S^0=K_1*K_2*S^0$ is a polytopal sphere too.
Take a vertex $v$ of $K_1$, $v$ is also a vertex of $K$.   Consider the link of $v$ in $K*S^0$, we have
$$\textrm{Link}_{K*S^0}(v)=\textrm{Link}_{K_1*K_2*S^0}(v)=\big(\textrm{Link}_{K_1}(v)\big)*K_2*S^0.$$
Using Lemmas \ref{N} and \ref{h}, since $K*S^0$ is polytopal sphere, both $\textrm{Link}_{K*S^0}(v)$
and $\big(\textrm{Link}_{K_1}(v)\big)*K_2$ are so.
Since $\dim\big(\textrm{Link}_{K_1}(v)\big)*K_2=\dim K_1*K_2-1=n$,
by induction hypothesis, $K_2$ is a polytopal sphere. In a similar way as above,
 choose a vertex of $K_2$, we can prove that $K_1$ is also a polytopal sphere.
\end{proof}

With  Proposition~\ref{s-poly} and Lemma~\ref{s} together, we have

\begin{theorem}\label{main1}
Let $P$ be an $n$-dimensional simple polytope with $m$ facets. Then $P$ is indecomposable (i.e., $P$ is not a product of polytopes) if and only if the polynomial $f_{K_P}=\sigma_n^{K_P}(x_1, ..., x_m)$ is indecomposable in $R[x_1, ..., x_m]$.
\end{theorem}

Associated with the equivariant characteristic classes of the moment-angle $\mathcal{Z}_P^{\Bbb F}$ over $P$, we have

\begin{corollary}
An $n$-dimensional simple polytope $P$ with $m$ facets is indecomposable  if and only if the $n$-th equivariant Stiefel--Whitney class
$w_n^{\mathbb{T}^m_{\Bbb R}}(\mathcal{T}\mathcal{Z}_P^{\Bbb R})$ or Chern class  $c_n^{\mathbb{T}^m_{\Bbb C}}(\mathcal{T}\mathcal{Z}_P^{\Bbb C})$ is  indecomposable in
$H^*(B\mathbb{T}_{\Bbb F}^m; R_{\Bbb F})=R_{\Bbb F}[x_1, ..., x_m]$.
\end{corollary}

\begin{example}
 $P$ is a prism which is the product $\Delta^1\times \Delta^2$ of two simplices and has 5 facets, shown as follows.
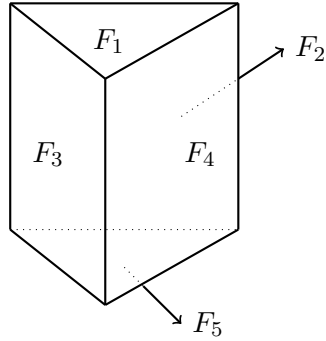
\begin{center}
		\begin{tikzpicture}
		\draw[thick,-] (-1,0) -- (2,0);
		\draw[thick,-] (-1,0) -- (.25,-1);
		\draw[thick,-] (2,0) -- (.25,-1);
		\draw[dotted,-] (-1,-3) -- (2,-3);
		\draw[thick,-] (-1,-3) -- (.25,-4);
		\draw[thick,-] (2,-3) -- (.25,-4);
		\draw[thick,-] (-1,0) -- (-1,-3);
		\draw[thick,-] (2,0) -- (2,-3);
		\draw[thick,-] (.25,-1) -- (.25,-4);

		\draw (.3,-.5) node[] {\small{$F_1$}};
		
		\draw (-.5,-2) node[] {\small{$F_3$}};
		
		\draw (1.5,-2) node[] {\small{$F_4$}};
		
		\draw[dotted,-] (0.5,-3.5) -- (0.75,-3.75);
		\draw[thick,->] (0.75,-3.75) -- (1.25,-4.25);
		\draw (1.25,-4.25) node[right] {\small{$F_5$}};
		\draw (2.6,-.6) node[right] {\small{$F_2$}};
		\draw[dotted,-] (1.25,-1.5) -- (2,-1);
		\draw[thick,->] (2,-1) -- (2.6,-.6);
		\end{tikzpicture}
	\end{center}
 Then we see that
\begin{eqnarray*}
 \sigma_3^{K_P}(x_1, ..., x_5)
&=&x_1x_2x_3+x_1x_2x_4+x_1x_3x_4+x_2x_3x_5+x_2x_4x_5+x_3x_4x_5\\
&=&(x_1+x_5)(x_2x_3+x_2x_4+x_3x_4)\\
\end{eqnarray*}
is decomposable in $R_{\Bbb F}[x_1, ..., x_5]$. However, if we cut out a vertex of $P$, then the resulting polytope
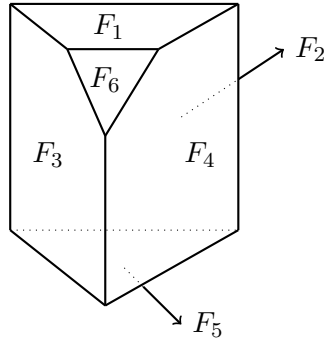
\begin{center}
       \begin{tikzpicture}

       %

       \draw[thick,-] (-1,0) -- (2,0);
       \draw[thick,-] (-1,0) -- (-0.25,-.6);
       \draw[thick,-] (2,0) -- (.95,-.6);
       \draw[thick,-] (-0.25,-.6) -- (.95,-.6);
       \draw[dotted,-] (-1,-3) -- (2,-3);
       \draw[thick,-] (-1,-3) -- (.25,-4);
       \draw[thick,-] (2,-3) -- (.25,-4);
       \draw[thick,-] (-1,0) -- (-1,-3);
       \draw[thick,-] (2,0) -- (2,-3);
       \draw[thick,-] (.25,-1.75) -- (.25,-4);
       \draw[thick,-] (-0.25,-.6) -- (.25,-1.75);
       \draw[thick,-] (.95,-.6) -- (.25,-1.75);

       \draw (.3,-.3) node[] {\small{$F_1$}};

       \draw (-.5,-2) node[] {\small{$F_3$}};

       \draw (1.5,-2) node[] {\small{$F_4$}};

       \draw[dotted,-] (0.5,-3.5) -- (0.75,-3.75);
       \draw[thick,->] (0.75,-3.75) -- (1.25,-4.25);
       \draw (1.25,-4.25) node[right] {\small{$F_5$}};
       \draw (2.6,-.6) node[right] {\small{$F_2$}};
       \draw[dotted,-] (1.25,-1.5) -- (2,-1);
       \draw[thick,->] (2,-1) -- (2.6,-.6);

       \draw (0.25,-1) node[] {\small{$F_6$}};

       \end{tikzpicture}
   \end{center}
is not a product, and  the corresponding $\sigma_3^{K_P}(x_1, ..., x_5, x_6)=x_1x_2x_3+x_1x_2x_4+x_1x_3x_6+x_1x_4x_6+x_3x_4x_6+x_2x_3x_5+x_2x_4x_5+x_3x_4x_5$
is indecomposable in $R_{\Bbb F}[x_1, ..., x_6]$.
\end{example}

\begin{remark}
We note that the decomposability of $\sigma_n^{K_P}(x_1, ..., x_m)$ in $R[x_1, ..., x_m]$ is different from that of $\sigma_n^{K_P}(x_1, ..., x_m)$ in $R(K_P)$.
In Subsection~\ref{exam} we shall give an example to show that the decomposability of $\sigma_n^{K_P}(x_1, ..., x_m)$ in $R(K_P)$
cannot be used to detect that of $P$.
\end{remark}

\section{An algebraic criterion for $P$ to be $n$-colorable}

\subsection{$n$-colorable polytopes}
Let $P$ be a simple polytope of dimension $n$ with $m$ facets. We say that $P$ is {\em $n$-colorable} if there is a coloring map $c: \mathcal{F}(P)\longrightarrow [n]=\{1, ..., n\}$ such that $c(F_i)\not=c(F_j)$ whenever $F_i\cap F_j\not=\varnothing$, where $\mathcal{F}(P)$ is the set of all facets of $P$. An equivalence definition for $P$ to be $n$-colorable is also given by saying that there is a non-degenerate simplicial map
$K_P\longrightarrow 2^{[n]}$, where $2^{[n]}$ is the complex determined by  all faces of the simplex $\Delta^{n-1}$  of dimension $n-1$.

\vskip .1cm

All $n$-colorable polytopes can produce an important class of quasitoric manifolds and small covers, which were studied and named as {\em pullbacks from the linear model} in \cite{dj}.
In terms of characteristic functions, it is easy to see that $P$ is $n$-colorable if and only if there is a characteristic function
$\lambda: \mathcal{F}(P)\longrightarrow R_{\Bbb F}^n$ such that the image $\im \lambda=\{e_1, ..., e_n\}$, where $\{e_1, ..., e_n\}$ is a basis of $R_{\Bbb F}^n$.

\vskip .1cm

In terms of combinatorics, Joswig gave the following criterion in~\cite{j}.

\begin{theorem}[\cite{j}]
An $n$-dimensional
simple polytope $P$ is $n$-colorable if and only if every 2-face
has an even number of edges.
\end{theorem}

More recently, a criterion in terms of self-dual binary codes of $P$ has been given in ~\cite{cly}.

\subsection{An algebraic criterion for $P$ to be $n$-colorable} Let $P$ be a simple polytope of dimension $n$ with $m$ facets.

\begin{theorem}\footnote[8]{We will see from the work of Notbohm~\cite{n1, n2} stated in Subsection~\ref{Notbohm} that Theorem~\ref{n-co} and Theorem~\ref{l-co} below give a weaker criterion for
 $P$ to be $\ell$-colorable than Notbohm's one, where $\ell\geq n$.}
\label{n-co}
The
simple polytope $P$ is $n$-colorable
if and only if $\sigma_n^{K_P}(x_1, ..., x_m)$ can be decomposed as   $\lambda_1\cdots\lambda_n$ of factors of degree $\dim {\Bbb F}$ in $R_{\Bbb F}(K_P)=R_{\Bbb F}[x_1, ..., x_m]/\mathcal{I}_{K_P}$.
\end{theorem}

\begin{proof}
First we note that there is a one-one correspondence between $\mathcal{F}(P)$ and $\{x_1, ..., x_m\}$. For convenience, for each facet $F\in \mathcal{F}(P)$, by $x_F$ we denote the corresponding element in $\{x_1, ..., x_m\}$, and conversely, for each element $x$ in $\{x_1, ..., x_m\}$, by $F_x$ we denote the corresponding facet in $\mathcal{F}(P)$.

\vskip .1cm
Assume that $P$ is $n$-colorable. Let $c: \mathcal{F}(P)\longrightarrow [n]$ be the coloring map of $P$. Then there is a partition $\{\mathcal{F}_1, ..., \mathcal{F}_n\}$ of $\mathcal{F}(P)$ such that
$\mathcal{F}_i=\{F\in \mathcal{F}(P)| c(F)=i\}$.  Now for each $\mathcal{F}_i$, set
$$\lambda_i=\sum_{F\in \mathcal{F}_i}x_F.$$
Since all facets of each $\mathcal{F}_i$ are disjoint, this means that  for each vertex $v$ of $P$, there are $n$ facets, say $F_1, ..., F_n$, which come from $\mathcal{F}_1, ..., \mathcal{F}_n$ respectively, such that
$v=F_1\cap\cdots\cap F_n$. Furthermore, we see easily that
$$\sigma_n^{K_P}(x_1, ..., x_m)=\lambda_1\cdots\lambda_n \text{ in } R_{\Bbb F}(K_P)$$
as desired.

\vskip .1cm

Conversely, assume that $\sigma_n^{K_P}(x_1, ..., x_m)=\lambda_1\cdots\lambda_n \text{ in } R(K_P)$, where $\deg \lambda_i=\dim {\Bbb F}$. Since $\sigma_n^{K_P}(x_1, ..., x_m)$ is sequare-free and each monomial of $\sigma_n^{K_P}(x_1, ..., x_m)$ has
 coefficient 1 and corresponds to a vertex of $P$, we see that
  \begin{enumerate}
 \item[(a)]
  any two different $\lambda_i, \lambda_j$  contain no the same monomials;
  \item[(b)] each $\lambda_i$ has the property that for any two monomials $y, y'$ in $\lambda_i$, $yy'=0$ in $R_{\Bbb F}(K_P)$;
   \item[(c)] the coefficients of all monomials of each $\lambda_i$  are 1 (this can easily be proved by performing an induction on $n$).
 \end{enumerate}
   For each $\lambda_i$, set
$$\mathcal{F}_i=\{F_x| x\text{ is a monomial of } \lambda_i\}.$$
Then $\{\mathcal{F}_1, ..., \mathcal{F}_n\}$ gives a partition of $\mathcal{F}(P)$ such that all facets of each $\mathcal{F}_i$ are disjoint. This determines a coloring map $c: \mathcal{F}(P)\longrightarrow [n]$ defined by
$c(F)=i$ if $F\in \mathcal{F}_i$. Thus, $P$ is $n$-colorable.
\end{proof}

\begin{corollary}
An $n$-dimensional
simple polytope $P$ is $n$-colorable
if and only if the $n$-th equivariant Stiefel--Whitney class
$w_n^{\mathbb{T}^m_{\Bbb R}}(\mathcal{T}\mathcal{Z}_P^{\Bbb R})$ or Chern class  $c_n^{\mathbb{T}^m_{\Bbb C}}(\mathcal{T}\mathcal{Z}_P^{\Bbb C})$ is a product
$\lambda_1\cdots \lambda_n$ in $H^*_{\mathbb{T}_{\Bbb F}^m}(\mathcal{Z}^{\Bbb F}_P; R_{\Bbb F})$, where $\lambda_i\in H^{\dim {\Bbb F}}_{\mathbb{T}_{\Bbb F}^m}(\mathcal{Z}^{\Bbb F}_P; R_{\Bbb F})$.
\end{corollary}

\begin{remark}
In Theorem~\ref{n-co}, $R_{\Bbb F}$ can be replaced by any commutative ring with unit.
\end{remark}

\subsection{An example of 3-colorable polytopes}\label{exam}

The following polytope $P$ with 14 facets  can be  colored by the coloring map $c:\mathcal{F}(P)\longrightarrow [3]=\{1, 2, 3\}$ satisfying that $c^{-1}(1)=\{F_1, F_6, F_7, F_8, F_9, F_{14}\}$, $c^{-1}(2)=\{F_2, F_4, F_{10}, F_{12}\}$, and  $c^{-1}(3)=\{F_3, F_5, F_{11}, F_{13}\}$. Thus this polytope $P$ is 3-colorable,  but obviously  it is not a product, so $\sigma_3^{K_P}(x_1, ..., x_{14})$ is indecomposable in the polynomial ring $R_{\Bbb F}[x_1,..., x_{14}]$.
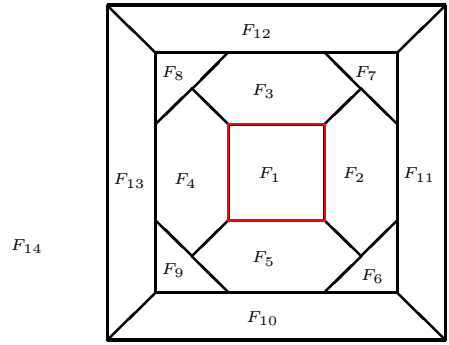
\[
\setlength{\unitlength}{1973sp}%
\begingroup\makeatletter\ifx\SetFigFont\undefined%
\gdef\SetFigFont#1#2#3#4#5{%
  \reset@font\fontsize{#1}{#2pt}%
  \fontfamily{#3}\fontseries{#4}\fontshape{#5}%
  \selectfont}%
\fi\endgroup%
\begin{picture}(5628,4244)(2551,-5783)
\put(7426,-3736){\makebox(0,0)[lb]{\smash{{\SetFigFont{6}{7.2}{\rmdefault}{\mddefault}{\updefault}{\color[rgb]{0,0,0}$F_{11}$}%
}}}}
\thicklines
\put(7351,-5161){\line( 1,-1){600}}
\put(7951,-5761){\line( 0, 1){4200}}
\put(7951,-1561){\line(-1,-1){600}}
\put(7351,-2161){\line( 0,-1){3000}}
\put(7351,-2161){\line( 1, 1){600}}
\put(7951,-1561){\line(-1, 0){4200}}
\put(3751,-1561){\line( 1,-1){600}}
\put(4351,-2161){\line( 1, 0){3000}}
\put(4351,-2161){\line(-1, 1){600}}
\put(3751,-1561){\line( 0,-1){4200}}
\put(3751,-5761){\line( 1, 1){600}}
\put(4351,-5161){\line( 0, 1){3000}}
\put(4351,-4261){\line( 0,-1){900}}
\put(4351,-5161){\line( 1, 0){900}}
\put(5251,-5161){\line(-1, 1){450}}
\put(4801,-4711){\line(-1, 1){450}}
\put(4351,-5161){\line(-1,-1){600}}
\put(3751,-5761){\line( 1, 0){4200}}
\put(7951,-5761){\line(-1, 1){600}}
\put(7351,-5161){\line(-1, 0){3000}}
\put(5251,-2161){\line(-1, 0){900}}
\put(4351,-2161){\line( 0,-1){900}}
\put(4351,-3061){\line( 1, 1){450}}
\put(4801,-2611){\line( 1, 1){450}}
\put(7351,-3061){\line( 0, 1){900}}
\put(7351,-2161){\line(-1, 0){900}}
\put(6451,-2161){\line( 1,-1){450}}
\put(6901,-2611){\line( 1,-1){450}}
\put(6451,-5161){\line( 1, 0){900}}
\put(7351,-5161){\line( 0, 1){900}}
\put(7351,-4261){\line(-1,-1){450}}
\put(6901,-4711){\line(-1,-1){450}}
\put(6451,-3061){\line( 0,-1){1200}}
\put(6451,-4261){\line( 1,-1){450}}
\put(6901,-4711){\line( 1, 1){450}}
\put(7351,-4261){\line( 0, 1){1200}}
\put(7351,-3061){\line(-1, 1){450}}
\put(6901,-2611){\line(-1,-1){450}}
\put(4351,-3061){\line( 0,-1){1200}}
\put(4351,-4261){\line( 1,-1){450}}
\put(4801,-4711){\line( 1, 1){450}}
\put(5251,-4261){\line( 0, 1){1200}}
\put(5251,-3061){\line(-1, 1){450}}
\put(4801,-2611){\line(-1,-1){450}}
\put(6451,-2161){\line(-1, 0){1200}}
\put(5251,-2161){\line(-1,-1){450}}
\put(4801,-2611){\line( 1,-1){450}}
\put(5251,-3061){\line( 1, 0){1200}}
\put(6451,-3061){\line( 1, 1){450}}
\put(6901,-2611){\line(-1, 1){450}}
\put(6451,-4261){\line(-1, 0){1200}}
\put(5251,-4261){\line(-1,-1){450}}
\put(4801,-4711){\line( 1,-1){450}}
\put(5251,-5161){\line( 1, 0){1200}}
\put(6451,-5161){\line( 1, 1){450}}
\put(6901,-4711){\line(-1, 1){450}}
\put(5626,-3736){\makebox(0,0)[lb]{\smash{{\SetFigFont{6}{7.2}{\rmdefault}{\mddefault}{\updefault}{\color[rgb]{0,0,0}$F_1$}%
}}}}
\put(6676,-3736){\makebox(0,0)[lb]{\smash{{\SetFigFont{6}{7.2}{\rmdefault}{\mddefault}{\updefault}{\color[rgb]{0,0,0}$F_2$}%
}}}}
\put(5551,-2686){\makebox(0,0)[lb]{\smash{{\SetFigFont{6}{7.2}{\rmdefault}{\mddefault}{\updefault}{\color[rgb]{0,0,0}$F_3$}%
}}}}
\put(4576,-3811){\makebox(0,0)[lb]{\smash{{\SetFigFont{6}{7.2}{\rmdefault}{\mddefault}{\updefault}{\color[rgb]{0,0,0}$F_4$}%
}}}}
\put(5551,-4786){\makebox(0,0)[lb]{\smash{{\SetFigFont{6}{7.2}{\rmdefault}{\mddefault}{\updefault}{\color[rgb]{0,0,0}$F_5$}%
}}}}
\put(6901,-5011){\makebox(0,0)[lb]{\smash{{\SetFigFont{6}{7.2}{\rmdefault}{\mddefault}{\updefault}{\color[rgb]{0,0,0}$F_6$}%
}}}}
\put(6826,-2461){\makebox(0,0)[lb]{\smash{{\SetFigFont{6}{7.2}{\rmdefault}{\mddefault}{\updefault}{\color[rgb]{0,0,0}$F_7$}%
}}}}
\put(4426,-4936){\makebox(0,0)[lb]{\smash{{\SetFigFont{6}{7.2}{\rmdefault}{\mddefault}{\updefault}{\color[rgb]{0,0,0}$F_9$}%
}}}}
\put(5476,-5536){\makebox(0,0)[lb]{\smash{{\SetFigFont{6}{7.2}{\rmdefault}{\mddefault}{\updefault}{\color[rgb]{0,0,0}$F_{10}$}%
}}}}
\put(5401,-1936){\makebox(0,0)[lb]{\smash{{\SetFigFont{6}{7.2}{\rmdefault}{\mddefault}{\updefault}{\color[rgb]{0,0,0}$F_{12}$}%
}}}}
\put(3826,-3811){\makebox(0,0)[lb]{\smash{{\SetFigFont{6}{7.2}{\rmdefault}{\mddefault}{\updefault}{\color[rgb]{0,0,0}$F_{13}$}%
}}}}
\put(2551,-4636){\makebox(0,0)[lb]{\smash{{\SetFigFont{6}{7.2}{\rmdefault}{\mddefault}{\updefault}{\color[rgb]{0,0,0}$F_{14}$}%
}}}}
\put(4426,-2461){\makebox(0,0)[lb]{\smash{{\SetFigFont{6}{7.2}{\rmdefault}{\mddefault}{\updefault}{\color[rgb]{0,0,0}$F_8$}%
}}}}
\thinlines
{\color[rgb]{1,0,0}\put(5251,-4261){\framebox(1200,1200){}}
}%
\end{picture}%

\]
However, in the Stanley--Reisner face ring $R_{\Bbb F}(K_P)=R_{\Bbb F}[x_1,..., x_{14}]/\mathcal{I}_{K_P}$,
\begin{eqnarray*}
 &&\sigma_3^{K_P}(x_1, ..., x_{14})\\
 &=& (x_1+x_6+x_7+x_8+x_9+x_{14})(x_2+x_4+x_{10}+x_{12})(x_3+x_5+x_{11}+x_{13})
 \end{eqnarray*}
which is decomposable.
 \subsection{More general case--$\ell$-colorable polytopes}
 Actually we can carry out our work on more general case.  Let $P$ be a simple polytope of dimension $n$ with $m$ facets, and let $\ell\geq n$ be an integer.  We say that $P$ is {\em $\ell$-colorable} if there is a coloring map $c: \mathcal{F}(P)\longrightarrow [\ell]=\{1, ..., \ell\}$ such that $c(F_i)\not=c(F_j)$ whenever $F_i\cap F_j\not=\varnothing$, where $\mathcal{F}(P)$ is the set of all facets of $P$.

 \vskip .1cm
 Equivalently, we can also say that  $P$ is $\ell$-colorable if and only if there is a  function
$\lambda: \mathcal{F}(P)\longrightarrow R_{\Bbb F}^\ell$ such that $\lambda(F_i)\not=\lambda(F_j)$ whenever $F_i\cap F_j\not=\varnothing$ and the image $\im \lambda=\{e_1, ..., e_\ell\}$, where $\{e_1, ..., e_\ell\}$ is a basis of $R_{\Bbb F}^\ell$.
We see that $P$ is always $m$-colorable. According to the construction method in~\cite[1.5. The basic construction]{dj},
we can use the pair $(P, \lambda)$ to construct a closed manifold $M_{\Bbb F}(P, \lambda)$ of dimension $n+(\dim {\Bbb F}-1)\ell$, with an $\mathbb{T}_{\Bbb F}^{\ell}$-action. In particular, if $\ell=m$, then
$M_{\Bbb F}(P, \lambda)$ is exactly the (real) moment-angle manifold $\mathcal{Z}_P^{\Bbb F}$.

\vskip .1cm

The following result is a generalization of Theorem~\ref{n-co}.
\begin{theorem} \label{l-co}
Let $P$ be a simple polytope of dimension $n$ with $m$ facets, and let $\ell\geq n$ be an integer.
Then
 $P$ is $\ell$-colorable
if and only if there exist $\ell$ homogenous  polynomial $\lambda_1, ..., \lambda_\ell$ such that  $\sigma_n^{K_P}(x_1, ..., x_m)=\sigma_n(\lambda_1, ...,\lambda_\ell)$ in $R(K_P)=R[x_1, ..., x_m]/\mathcal{I}_{K_P}$, where
each $\lambda_i$ is a linear combination of $x_1, ..., x_m$
\end{theorem}

\begin{proof}
The proof follows closely that of Theorem~\ref{n-co} in a very similar way. We would
like to leave it as an exercise to the reader.
\end{proof}

\subsection{The work of  Notbohm in~\cite{n1, n2}}\label{Notbohm}

For an $n$-colorable polytope $P$ with $m$ facets, Davis and Januszkiewicz in~\cite[6.2]{dj} studied a pullback $\pi_{\Bbb F}: M_{\Bbb F}\longrightarrow P$ of the linear model, so the image of the characteristic function
$\lambda: \mathcal{F}(P)\longrightarrow R_{\Bbb F}^n$ is some basis  $\{e_1, ..., e_n\}$ of $R_{\Bbb F}^n$.   Associated with the bundle $\mathbb{L}_1\oplus\cdots\oplus \mathbb{L}_m$ over $E\mathbb{T}_{\Bbb F}^m\times_{\mathbb{T}_{\Bbb F}^m}\mathcal{Z}^{\Bbb F}_P$, they used the sequence $\lambda=(\lambda_1, ..., \lambda_n)$  to construct $n$ line bundles $E_i=\mathbb{L}_1^{\otimes \lambda_{i1}}\otimes\cdots\otimes
 \mathbb{L}_m^{\otimes \lambda_{im}}, i=1, ..., n,$ with the first characteristic class of each $E_i$ just being $\lambda_i$ such that
$$E_1\oplus\cdots\oplus E_n\oplus\underline{\Bbb F}^{m-n}\cong \mathbb{L}_1\oplus\cdots\oplus \mathbb{L}_m $$
where each $\lambda_i=\lambda_{i1}x_1+\cdots+\lambda_{im}x_m\in R_{\Bbb F}(K_P)=R_{\Bbb F}[x_1, ..., x_m]/\mathcal{I}_{K_P}$ and $\underline{\Bbb F}^{m-n}$ denotes the trivial vector bundle of dimension $m-n$. In~\cite{n1, n2}, Notbohm generalized this to the study of the vector bundle over  Davis--Januszkiewicz space $DJ(K)$  for an $(n-1)$-dimensional simplicial complex $K$. Notbohm constructed a
particular $n$-dimensional complex  vector bundle over the associated Davis--Januszkiewicz space whose Chern classes are given by the elementary symmetric
polynomials in the generators of the Stanley--Reisner algebra. Furthermore, Notbohm showed
that the isomorphism type of this complex vector bundle as well as of its realification
are completely determined by its characteristic classes. As an application, Notbohm proved
 that coloring properties of the simplicial complex are completely determined by
splitting properties of this bundle.

\vskip .1cm

Together with Theorems 1.1-1.2 and Corollary 1.8 in~\cite{n2}, Notbohm gave the following theorem
\begin{theorem}[Notbohm]\label{Notb}
Let $P$ be a simple polytope of dimension $n$ with $m$ facets, and let $\ell\geq n$ be an integer.
The following statements are equivalent.
\begin{enumerate}
 \item
 $P$ is $\ell$-colorable.
 \item  There are $\ell$ line bundles $E_i (i=1, ..., \ell)$ such that
$$E_1\oplus\cdots\oplus E_\ell\oplus\underline{\Bbb F}^{m-\ell}\cong \mathbb{L}_1\oplus\cdots\oplus \mathbb{L}_m. $$
\item $c_{\Bbb F}(E_1\oplus\cdots\oplus E_\ell)=c_{\Bbb F}(\mathbb{L}_1\oplus\cdots\oplus \mathbb{L}_m)$ where $c_{\Bbb F}$ denotes the total Chern class if ${\Bbb F}={\Bbb C}$ and the total Stiefel--Whitney class if ${\Bbb F}={\Bbb R}$.
\end{enumerate}
\end{theorem}

\begin{remark}
In~Theorem~\ref{Notb}(3), write $c_{\Bbb F}(E_i)=1+\lambda_i$ where $\lambda_i\in R(K_P)=R[x_1, ..., x_m]/\mathcal{I}_{K_P}$ with $\deg\lambda_i=\dim{\Bbb F}$. Then we have that
$$\prod_{i=1}^\ell(1+\lambda_i)=c_{\Bbb F}(E_1\oplus\cdots\oplus E_\ell)=c_{\Bbb F}(\mathbb{L}_1\oplus\cdots\oplus \mathbb{L}_m)=\prod_{j=1}^m(1+x_j)$$ so for $i\leq n$,
$$\sigma_i^{K_P}(x_1, ..., x_m)=\sigma_i(\lambda_1, ...,\lambda_\ell).$$
This means  that in Theorem~\ref{Notb}, all $\sigma_i^{K_P}(x_1, ..., x_m), i\leq n$ would be involved for $P$ to be $\ell$-colorable. However, Theorems~\ref{n-co}--~\ref{l-co} tell us that we only need to consider the $n$-th elementary symmetric polynomial $\sigma_n^{K_P}(x_1, ..., x_m)$.
\end{remark}

\section{Stanley--Reisner  exterior face ring and Buchstaber invariant}

\subsection{Stanley--Reisner  exterior face ring}\label{exterior}
Let $K$ be a simplicial complex on vertex set $[m]=\{1, ..., m\}$. Consider the exterior algebra $\wedge_{{\Bbb Z}}[x_1, ..., x_m]$  over ${\Bbb Z}$. In a similar way to Stanley--Reisner face ring, based upon the combinatorial structure of $K$, we may define the quotient ring of  the exterior algebra $\wedge_{{\Bbb Z}}[x_1, ..., x_m]$
$$\mathcal{E}_{\Bbb C}(K)=\wedge_{{\Bbb Z}}[x_1, ..., x_m]/\mathcal{I}_{K}$$
where $\mathcal{I}_{K}$ is  the ideal generated by those square-free monomials $x_{i_1}\wedge\cdots\wedge x_{i_r}$ for which $\{i_1, ...,  i_r\}\not\in K$.
We call $\mathcal{E}_{\Bbb C}(K)$ the {\em Stanley--Reisner exterior face ring of $K$}.

\vskip .1cm

Clearly, each polynomial in $\mathcal{E}_{\Bbb C}(K)$ is square-free.
We can define a differential $\partial$ on $\mathcal{E}_{\Bbb C}(K)$ in a usual way. Then we have
\begin{lemma}\label{orient}
 $(\mathcal{E}_{\Bbb C}(K), \partial)$ is isomorphic to the chain complex $\mathcal{C}(K)$ of $K$.
\end{lemma}
\begin{proof}
Given an orientation of $K$,  each oriented simplex $a=\langle i_1, ..., i_r\rangle$ in $K$ determines a monomial ${\bf x}_a=x_{i_1}\wedge\cdots\wedge x_{i_r}$ in
$\mathcal{E}_{\Bbb C}(K)$ such that the ordering of the product $x_{i_1}\wedge\cdots\wedge x_{i_r}$ agrees with the orientation of $a$. Then the required isomorphism is given by mapping $a$ to ${\bf x}_a$.
\end{proof}

In the proof of Lemma~\ref{orient}, we call ${\bf x}_a$ an {\em oriented monomial} corresponding to the oriented simplex $a$.

\vskip .1cm

Set $\mathcal{E}_{\Bbb R}(K)=\mathcal{E}_{\Bbb C}(K)\otimes {\Bbb Z}_2$. This is the Stanley--Reisner exterior face ring over ${\Bbb Z}_2$, which is commutative. We will take $\deg x_i=\dim {\Bbb F}$ for the generators $x_i$ of $\mathcal{E}_{\Bbb F}(K)$ in the following discussion.

\vskip .1cm

Compare with Stanley--Reisner  face ring $R_{\Bbb F}(K)$, if ${\Bbb F}=\R$, then it is easy to see that there is a natural module embedding $\mathcal{E}_{\Bbb R}(K)\hookrightarrow R_{\Bbb R}(K)$, so  $\mathcal{E}_{\Bbb R}(K)$ can be regarded as a submodule of $R_{\Bbb R}(K)$. On the other hand, there is a surjective projection $R_{\Bbb R}(K)\longrightarrow \mathcal{E}_{\Bbb R}(K)$. This projection is a ring homomorphism. Thus $\mathcal{E}_{\Bbb R}(K)$ is isomorphic to the quotient
$R_{\Bbb R}(K)/\langle x_i^2|i=1, ..., m\rangle$ as rings.

\vskip .1cm

If ${\Bbb F}=\C$, this case depends upon a choice of orientations. Given an orientation of $K$,
consider all oriented monomials $x_{i_1}\wedge\cdots\wedge x_{i_r}$ coreesponding to all oriented simplices in $K$.  These oriented monomials form a basis of
$\mathcal{E}_{\Bbb C}(K)$ as a module. Fix this basis,  then there is still a module embedding $\mathcal{E}_{\Bbb C}(K)\hookrightarrow R_{\Bbb C}(K)$ defined by mapping
$x_{i_1}\wedge\cdots\wedge x_{i_r}$ to $x_{i_1}\cdots x_{i_r}$. Of course, we can also define a projection $R_{\Bbb C}(K)\longrightarrow \mathcal{E}_{\Bbb C}(K)$ by mapping non sequare-free polynomials to zero and mapping
sequare free monomials to corresponding elements of the basis.
However, this projection is only a module homomorphism rather than a ring homomorphism. Thus, as graded modules, $\mathcal{E}_{\Bbb C}(K)$ is isomorphic to $R_{\Bbb C}(K)/\langle x_i^2|i=1, ..., m\rangle$.
In particular, $\mathcal{E}_{\Bbb C}(K)^{(2)}\cong R_{\Bbb C}(K)^{(2)}$, where $\mathcal{E}_{\Bbb C}(K)^{(2)}$ and  $R_{\Bbb C}(K)^{(2)}$ are the graded modules of degree 2 in $\mathcal{E}_{\Bbb C}(K)$ and  $R_{\Bbb C}(K)$, respectively.

\subsection{A criterion for Buchstaber invariant $s_{\Bbb C}(P)=m-n$}
Let $P$ be an $n$-dimensional simple polytope with $m$ facets $F_1, ..., F_m$. In the Stanley--Reisner exterior face ring $\mathcal{E}_{\Bbb C}(K_P)=\wedge_{\Z}[x_1, ..., x_m]/\mathcal{I}_{K_P}$, we can still define elementary symmetric polynomials $\sigma_i^{K_P}(x_1, ..., x_m)$, but here our definition of $\sigma_i^{K_P}(x_1, ..., x_m)$ depends upon the choices of the orientations of $K_P$.

\vskip .1cm

Given an orientation $\mathfrak{o}$ of $K_P$, define the $i$-th elementary symmetric polynomial as
$$\sigma_i^{K_P, \mathfrak{o}}(x_1, ..., x_m)=\sum_{a\in K_P \atop
\dim a=i-1}{\bf x}_a$$
where ${\bf x}_a$ is the oriented monomial of degree $2i$ corresponding to $a$, as defined as before.

\vskip .1cm

We shall show the following necessary and sufficient condition for $s_{\Bbb C}(P)=m-n$ in terms of the decomposability of $\sigma_n^{K_P, \mathfrak{o}}(x_1, ..., x_m)$.

\begin{theorem}\label{cpx}
The Buchstaber invariant $s_{\Bbb C}(P)=m-n$ if and only if there is an orientation $\mathfrak{o}$ of $K_P$ such that $\sigma_n^{K_P, \mathfrak{o}}(x_1, ..., x_m)$ can be decomposed as the product $\lambda_1\wedge\cdots\wedge\lambda_n$ of $n$ factors of degree 2 in $\mathcal{E}_{\Bbb C}(K_P)$.
\end{theorem}
\begin{proof}
We firstly choose the orientation $\mathfrak{o'}$ of $K_P$ which agrees with the ordering of its vertex set $[m]$ (corresponding to the facet set $\{F_1, ..., F_m\}$ of $P$).

\vskip .1cm

Suppose that $s_{\Bbb C}(P)=m-n$ and $\lambda$ is a characteristic function on $P$. As mentioned in Subsection~\ref{char}, it determines an $(n\times m)$-matrix $\Lambda=(\lambda_{ij})$ and a sequence $\lambda=(\lambda_1, ..., \lambda_n)$ in $R_{\Bbb C}(K_P)^{(2)}$. For any oriented $(n-1)$-simplex $a=\langle i_1, ..., i_n\rangle\in K_P$ with $i_1<\cdots <i_n$,  denote by $\Lambda_a=(\lambda(F_{i_1}), ..., \lambda(F_{i_n}))$ the corresponding $(n\times n)$-matrix . Then $\det\Lambda_a=\pm 1$. Now regarding $\lambda_1, ..., \lambda_n$ as the elements in $\mathcal{E}_{\Bbb C}(K_P)^{(2)}$ since $R_{\Bbb C}(K_P)^{(2)}\cong \mathcal{E}_{\Bbb C}(K_P)^{(2)}$, let us consider the product $\lambda_1\wedge\cdots\wedge\lambda_n$ in $\mathcal{E}_{\Bbb C}(K_P)$. Since this product $\lambda_1\wedge\cdots\wedge\lambda_n$ is of degree  $2n$, it is actually the linear combination of ${\bf x}_a$ with integer coefficients, where $a$ runs over all oriented $(n-1)$-simplices of $K_P$ with the given orientation $\mathfrak{o'}$. We only need to figure out the coefficient for every ${\bf x}_a$.

Without the loss of generality, suppose that $a_0=\langle 1, ..., n\rangle\in K_P$. Then $\lambda_i=\lambda_{i1}x_1+\cdots+\lambda_{in}x_n+\cdots+\lambda_{im}x_m$.
\begin{equation}\label{wedge}
\lambda_1\wedge\cdots\wedge\lambda_n=\sum_{\binom{1, ..., n}{i_1, ..., i_n}\in S_n}(\lambda_{1i_1}x_{i_1})\wedge\cdots\wedge(\lambda_{ni_n}x_{i_n})+f_1
\end{equation}
where $S_n$ is the symmetric group of rank $n$ and $f_1$ is the part of $\lambda_1\wedge\cdots\wedge\lambda_n$ which doesn't contain the monomial ${\bf x}_{a_0}$. Denote by $f_2$ the 1st part of the right side of ~\ref{wedge}. Then
\begin{eqnarray*}
  f_2 &=& \sum \lambda_{1i_1}\cdots\lambda_{ni_n}\cdot x_{i_1}\wedge\cdots\wedge x_{i_n} \\
      &=& \left(\sum \text{sign}\binom{1, ..., n}{i_1, ..., i_n}\lambda_{1i_1}\cdots\lambda_{ni_n}\right)x_1\wedge\cdots\wedge x_n \\
      &=& \det\Lambda_{a_0}\cdot {\bf x}_{a_0}.
\end{eqnarray*}
Similarly, we can know that the coefficient of other ${\bf x}_a$ also equals to $\det\Lambda_a$, which is just $\pm 1$. So we only need to choose a proper orientation $\mathfrak{o}$ of $K_P$ such that all coefficients of ${\bf x}_a$ corresponding to $a$ under the chosen orientation $\mathfrak{o}$ equal to $1$. In fact, if $\det\Lambda_a=-1$, then we will choose the inverse orientation of $a$. Furthermore, under the chosen orientation $\mathfrak{o}$, we have that $\sigma_n^{K_P, \mathfrak{o}}(x_1, ..., x_m)=\lambda_1\wedge\cdots\wedge\lambda_n$.

\vskip .1cm

Next, let us prove the converse. Assume that there is a proper orientation $\mathfrak{o}$ of $K_P$ such that $\sigma_n^{K_P, \mathfrak{o}}(x_1, ..., x_m)=\lambda_1\wedge\cdots\wedge\lambda_n$.
Let $\lambda_i=\lambda_{i1}x_1+\cdots+\lambda_{im}x_m$. From the above arguments, we  know that under the orientation $\mathfrak{o'}$,
$$\lambda_1\wedge\cdots\wedge\lambda_n=\sum \det\Lambda_{a^{\mathfrak{o'}}} {\bf x}_{a^{\mathfrak{o'}}}$$
where
$\Lambda_{a^{\mathfrak{o'}}}=
\begin{pmatrix}
\begin{smallmatrix}
\lambda_{1i_1}&\cdots & \lambda_{1i_n}\\
\vdots& & \vdots\\
\lambda_{ni_1}& \cdots &\lambda_{ni_n}
\end{smallmatrix}
\end{pmatrix}$
for $a^{\mathfrak{o'}}=\langle i_1, ..., i_n\rangle^{\mathfrak{o'}}\in K_P$ with $i_1<\cdots<i_n$. On the other hand,  under the orientation $\mathfrak{o}$, we have that
$$\sigma_n^{K_P, \mathfrak{o}}(x_1, ..., x_m)=\lambda_1\wedge\cdots\wedge\lambda_n=\sum {\bf x}_{a^{\mathfrak{o}}}.$$
Two different expressions of the product $\lambda_1\wedge\cdots\wedge\lambda_n$ under the two orientations $\mathfrak{o'}$ and $\mathfrak{o}$ give
$${\bf x}_{a^{\mathfrak{o}}}=\det\Lambda_{a^{\mathfrak{o'}}} {\bf x}_{a^{\mathfrak{o'}}}$$
where two oriented $(n-1)$-dimensional simplices $a^{\mathfrak{o}}$ and $a^{\mathfrak{o'}}$ have the same underlying set. This means that  $\det\Lambda_{a^{\mathfrak{o'}}}=\pm 1$ for every $(n-1)$-dimensional oriented simplex $a^{\mathfrak{o'}}$ under the  orientation $\mathfrak{o'}$ since ${\bf x}_{a^{\mathfrak{o}}}=\pm {\bf x}_{a^{\mathfrak{o'}}}$ in $\mathcal{E}_{\Bbb C}(K_P)$. Then $\Lambda=(\lambda_{ij})$ can just determine a characteristic function and therefore, $s_{\Bbb C}(P)=m-n$.
\end{proof}

\begin{remark} Let $V_P$ denote the vertex set of $P$. Then we know that there are $|V_P|$ simplices of dimension $n-1$ in $K_P$.
Thus the number of  possible orientations of all simplices of dimension $n-1$ in $K_P$ is $2^{|V_P|}$.
It should be pointed out that if $s_{\Bbb C}(K_P)=m-n$, as shown in the proof of Theorem~\ref{cpx}, we see that for any orientation $\mathfrak{o}$ of $K_P$,  $\sigma_n^{K_P, \mathfrak{o}}(x_1, ..., x_m)$ may not be decomposed as the product $\lambda_1\wedge\cdots\wedge\lambda_n$ of $n$ factors of degree 2 in $\mathcal{E}_{\Bbb C}(K_P)$. Actually, this can be related with the geometry structure of the quasitoric manifold corresponding to $\lambda=(\lambda_1, ..., \lambda_n)$ with the given omniorientation (see Proposition~\ref{al-com} below).
\end{remark}

\begin{corollary}\label{ch-pol1}
A sequence $\lambda=(\lambda_1, ..., \lambda_n)$ in $R_{\C}(K_P)^{(2)}$ is a characteristic function on $P^n$ if and only if there is an orientation $\mathfrak{o}$ of $K_P$ such that $\sigma_n^{K_P, \mathfrak{o}}(x_1, ..., x_m)=\lambda_1\wedge\cdots\wedge\lambda_n$  in $\mathcal{E}_{\Bbb C}(K_P)$.
\end{corollary}


\begin{example} \label{almost complex}
Let $P$ be a square with 4 facets $F_1, F_2, F_3, F_4$ satisfying that $F_1\cap F_3=\emptyset$ and $F_2\cap F_4=\emptyset$. Then $K_P$ is a 1-dimensional simplicial complex on $[4]=\{1, 2, 3, 4\}$. Given an orientation $\mathfrak{o}$ of $K_P$ such that 4 oriented 1-dimensional simplices are $\langle1,2\rangle, \langle1, 4\rangle, \langle3,2\rangle, \langle4,3\rangle$, respectively,  then
$$\sigma_2^{K_P, \mathfrak{o}}(x_1, x_2, x_3, x_4)=x_1\wedge x_2+x_1\wedge x_4+x_3\wedge x_2+x_4\wedge x_3.$$
Choose $\lambda_1=x_1+x_2+2x_3+3x_4$ and $\lambda_2=2x_1+3x_2+5x_3+7x_4$. Then we have that
\begin{eqnarray*}
\lambda_1\wedge\lambda_2 &= &(x_1+x_2+2x_3+3x_4)\wedge(2x_1+3x_2+5x_3+7x_4)\\
&=& 3x_1\wedge x_2+7x_1\wedge x_4+2x_2\wedge x_1+5x_2\wedge x_3+6x_3\wedge x_2+14x_3\wedge x_4\\
&& +15x_4\wedge x_3+6x_4\wedge x_1 \\
&=& 3x_1\wedge x_2+7x_1\wedge x_4-2x_1\wedge x_2+5x_2\wedge x_3-6x_2\wedge x_3+14x_3\wedge x_4\\
&& -15x_3\wedge x_4-6x_1\wedge x_4\\
&=& x_1\wedge x_2+x_1\wedge x_4+x_3\wedge x_2+x_4\wedge x_3
\end{eqnarray*}
so $\sigma_2^{K_P, \mathfrak{o}}(x_1, x_2, x_3, x_4)=\lambda_1\wedge\lambda_2$.  Thus $s_{\Bbb C}(P)=4-2=2$ and $\lambda=(\lambda_1, \lambda_2)$ is a characteristic function on $P$.
\end{example}

\subsection{Almost complex structures on quasitoric manifolds}

As a manifold with corners, $P$ has two possible orientations, each of which determines an orientation  of $K_P$.
If we take such an orientation $\mathfrak{o}_P$ of $P$, then the other one will be the inverse orientation $-\mathfrak{o}_P$ of $\mathfrak{o}_P$.
These two orientations just determine two possible orientations of the underlying space $|K_P|$ as a sphere, denoted by $\pm \mathfrak{o}_{|K_P|}$.

\vskip .1cm

Now take the orientation  $\mathfrak{o}_{|K_P|}$ of $K_P$ and assume that $s_{\Bbb C}(P)=m-n$. Let $\lambda=(\lambda_1, ..., \lambda_n)$ be a characteristic function on $P$ such that $\lambda_i=\lambda_{i1}x_1+\cdots+\lambda_{im}x_m$, and  let $M_{\Bbb C}(P,\lambda)$ be the corresponding quasitoric manifold over $P$.
For each oriented $(n-1)$-simplex $a=\langle i_1, ..., i_n\rangle$ of
$K_P$, by $\Lambda_a$ we denote the sequare matrix formed by $n$ columns $\lambda(F_{i_1}), ..., \lambda(F_{i_n})$ in the order given by the orientation of $a$, where $\lambda(F_{i_j})=(\lambda_{1i_j}, ..., \lambda_{ni_j})^\top$. Note that each oriented simplex $a$ of dimension $n-1$ in $K_P$ corresponds to a fixed point of  $M_{\Bbb C}(P,\lambda)$, denoted by  $v_a$.
Then we have that
$$\lambda_1\wedge\cdots\wedge\lambda_n=\sum_{a\in K_P \atop
\dim a=n-1}\det \Lambda_a{\bf x}_a.$$
Without the loss of generality, assume that $\langle 1, ..., n\rangle$ is an oriented $(n-1)$-simplex of $K_P$ and $\lambda_{\langle 1, ..., n\rangle}$ is the  identity $(n\times n)$-matrix. By \cite[Definition 5.3]{bpr}, $(P, \lambda)$
is a combinatorial quasitoric pair, and it determines an omniorientation of $M_{\Bbb C}(P,\lambda)$.
However,  in general, $\lambda_1\wedge\cdots\wedge\lambda_n=\sum_{a\in K_P \atop
\dim a=n-1}\det \Lambda_a{\bf x}_a$ may not be equal to $\sigma_n^{K_P, \mathfrak{o}_{|K_P|}}(x_1, ..., x_m)$. Indeed, we easily see that the coefficient $\det \Lambda_a$ is exactly the sign $\sigma(v_a)$ at the fixed point $v_a$
(see~\cite[Definition B.6.2]{bp1} for the definition of $\sigma(v_a)$), which may be equal to 1 or $-1$.
\cite[Theorem 7.3.24]{bp1} tells us that $M_{\Bbb C}(P,\lambda)$ with the given omniorientation admits a $\mathbb{T}_{\Bbb C}^n$-invariant almost complex structure
if and only if $\sigma(v)=1$ for all fixed points of $M_{\Bbb C}(P,\lambda)$. Thus by Theorem~\ref{cpx} we have

\begin{proposition}\label{al-com}
$\sigma_n^{K_P, \mathfrak{o}_{|K_P|}}(x_1, ..., x_m)$ can be decomposed as $\lambda_1\wedge\cdots\wedge\lambda_n$ in  $\mathcal{E}_{\Bbb C}(K_P)$ if and only if
$s_{\Bbb C}(P)=m-n$ and the corresponding quasitoric manifold $M_{\Bbb C}(P,\lambda)$ admits a $\mathbb{T}_{\Bbb C}^n$-invariant almost complex structure.
 \end{proposition}

 Fix the orientation  $\mathfrak{o}_{|K_P|}$ of $K_P$,  we easily see that $\sigma_n^{K_P, \mathfrak{o}_{|K_P|}}(x_1, ..., x_m)$ is a cycle in the chain complex $(\mathcal{E}_{\Bbb C}(K_P), \partial)$, i.e., $\partial \sigma_n^{K_P, \mathfrak{o}_{|K_P|}}(x_1, ..., x_m)=0$. An easy argument shows that if $\lambda_1\wedge\cdots\wedge\lambda_n$ is a cycle in  $(\mathcal{E}_{\Bbb C}(K_P), \partial)$, then  $\sigma_n^{K_P, \mathfrak{o}_{|K_P|}}(x_1, ..., x_m)
 =\pm \lambda_1\wedge\cdots\wedge\lambda_n$ since  $(\mathcal{E}_{\Bbb C}(K_P), \partial)\cong \mathcal{C}(K_P)$ as chain complexes by Lemma~\ref{orient} and $K_P$ is a simplicial sphere. Therefore, it follows that

 \begin{corollary}\label{cycle}
Let $\lambda=(\lambda_1, ..., \lambda_n)$ be a characteristic function on $P$ and let $M_{\Bbb C}(P,\lambda)$ be the corresponding quasitoric manifold over $P$.
 Then $M_{\Bbb C}(P,\lambda)$ admits a $\mathbb{T}_{\Bbb C}^n$-invariant almost complex structure if and only if
 $\lambda_1\wedge\cdots\wedge\lambda_n$ is a cycle in  $\mathcal{E}_{\Bbb C}(K_P)$.
 \end{corollary}

 \begin{example}
In Example~\ref{almost complex}, we give an orientation of $P$ such that all 1-dimensional oriented simplices of $K_P$ are $\langle 1, 2\rangle, \langle 2, 3\rangle, \langle 3, 4\rangle, \langle 4, 1\rangle$, respectively. We still choose
$\lambda=(\lambda_1, \lambda_2)$ such that  $\lambda_1=x_1+x_2+2x_3+3x_4$ and $\lambda_2=2x_1+3x_2+5x_3+7x_4$, as stated in Example~\ref{almost complex}. Then we know that $$\lambda_1\wedge\lambda_2 = x_1\wedge x_2-x_2\wedge x_3-x_3\wedge x_4-x_4\wedge x_1,$$ but obviously
 $\lambda_1\wedge\lambda_2$ is not a cycle in $\mathcal{E}_{\Bbb C}(K_P)$. Thus the corresponding 4-dimensional quasitoric manifold $M_{\Bbb C}(P,\lambda)$ does not admit a $\mathbb{T}_{\Bbb C}^4$-invariant almost complex structure.

 \vskip .1cm

 However, if we choose $\lambda'=(\lambda_1', \lambda_2')$ such that $\lambda_1'=x_1-x_3$ and $\lambda_2'=x_2-x_4$, then
 $$\lambda_1'\wedge\lambda_2'=x_1\wedge x_2+x_2\wedge x_3+x_3\wedge x_4+x_4\wedge x_1$$
 so $\lambda_1'\wedge\lambda_2'$ is a cycle in $\mathcal{E}_{\Bbb C}(K_P)$. Thus the corresponding 4-dimensional quasitoric manifold $M_{\Bbb C}(P,\lambda')$  admits a $\mathbb{T}_{\Bbb C}^4$-invariant almost complex structure.
\end{example}

 \subsection{A criterion for Buchstaber invariant $s_{\Bbb R}(P)=m-n$}
 In this case, we know from Subsection~\ref{exterior} and Theorem~\ref{face ring} that there is a natural surjective projection
 $$H^*_{T^m_{\Bbb R}}(\mathcal{Z}_P^{\Bbb R}; {\Bbb Z}_2)=R_{\R}(K_P)\longrightarrow \mathcal{E}_{\Bbb R}(K_P)=H^*_{T^m_{\Bbb R}}(\mathcal{Z}_P^{\Bbb R}; {\Bbb Z}_2)/\langle x_i^2|i=1, ..., m\rangle$$
where $\mathcal{E}_{\Bbb R}(K_P)=\wedge_{{\Bbb Z}_2}[x_1, ..., x_m]/\mathcal{I}_{K_P}$ with $\deg x_i=1$.

\vskip .1cm

There is a real analogue of Theorem~\ref{cpx}. Since we do not need to consider the orientation in the real case, its proof is much easier and omitted.

\begin{theorem}\label{real}
 $s_{\Bbb R}(P)=m-n$ if and only if
$w_n^{\mathbb{T}^m_{\Bbb R}}(\mathcal{T}\mathcal{Z}_P^{\Bbb R})=\sigma_n^{K_P}(x_1, ..., x_m)$ can be decomposed as the product
$\lambda_1\cdots\lambda_n$ of $n$ factors of degree 1  in $\mathcal{E}_{\Bbb R}(K_P)=H^*_{T^m_{\Bbb R}}(\mathcal{Z}_P^{\Bbb R}; {\Bbb Z}_2)/\langle x_i^2|i=1, ..., m\rangle$.
\end{theorem}

\begin{corollary}\label{ch-pol2}
The set $$\{\lambda=(\lambda_1, ..., \lambda_n)\text{ in }H^1_{\mathbb{T}^m_{\Bbb R}}(\mathcal{Z}_P^{\Bbb R}; {\Bbb Z}_2)\big|\sigma_n^{K_P}(x_1, ..., x_m)=\lambda_1\cdots\lambda_n \text{ in } \mathcal{E}_{\Bbb R}(K_P)\}$$
gives all characteristic functions on $P^n$.
\end{corollary}

\subsection{An application to cyclic polytopes}

The moment curve in $\R^n$ is given by $v(t)=(t, t^2, ..., t^n)$. For any $m>n$,  define the cyclic polytope $C^n(m)$ as the convex hull of $m$ distinct points $v(t_i)$, $t_1<t_2<\cdots<t_m$ on the moment curve in $\R^n$. $C^n(m)$ is a simplicial polytope, i.e. its boundary is a simplicial $(n-1)$-sphere, still denoted by $C^n(m)$.

\vskip .1cm
In the case of $n\geq 4$, it has been already shown in \cite{dj} that when $m\ge 2^n$, $s_{\Bbb F}(C^n(m))<m-n$, since $C^n(m)$ is 2-neighbourly. But actually the precise lower bound for $m$ is much smaller.

\vskip .1cm
First let us look at the case $n=4$.

\begin{lemma}\label{c4}
When $m\ge 8$, $s_{\Bbb R}(C^4(m))<m-4$. In particular, when $m=8$, $s_{\Bbb R}(C^4(m))=m-5$, and when $m=7$, $s_{\Bbb R}(C^4(m))=m-4$.
  \end{lemma}
 \begin{proof} For a convenience, write $v_i=v(i)$ where $i=1, ..., m$.  According to \cite{s}, if $a=\{v_{i_1}, v_{i_2}, v_{i_3}, v_{i_4}\}\in C^4(m)$ where $i_1<\cdots<i_4$, then there are exactly 2 possibilities:
\begin{enumerate}
  \item $i_1=1$, $i_4=m$, $i_3-i_2=1$.
  \item $i_2-i_1=i_4-i_3=1$.
\end{enumerate}
Suppose there exists a characteristic function $\lambda: \{v_i\}_{1\le i\le m}\to \Z_2^4$. Then by Theorem~\ref{real},
$$\sigma_4^{C^4(m)}(x_1, ..., x_m)=\lambda_1\cdots\lambda_4,$$
where $\lambda_i=\lambda_{i1}x_1+\cdots+\lambda_{im}x_m$ are elements of degree 1 in $\mathcal{E}_{\Bbb R}(C^4(m))$.
\vskip .1cm

We claim:
\begin{enumerate}
  \item $(\lambda_{1i}, \cdots, \lambda_{4i})\ne(\lambda_{1j}, \cdots, \lambda_{4j})$, if $i\ne j$.
  \item $(\lambda_{1i}, \cdots, \lambda_{4i})+(\lambda_{1j}, \cdots, \lambda_{4j})\ne(\lambda_{1k}, \cdots, \lambda_{4k})$, if $i-j\equiv 1$ or $m-1\pmod{m}$, and $k\ne i,j$.
  \item $(\lambda_{1i}, \cdots, \lambda_{4i})+(\lambda_{1j}, \cdots, \lambda_{4j})\ne(\lambda_{1k}, \cdots, \lambda_{4k})+(\lambda_{1l}, \cdots, \lambda_{4l})$, if $i-j\equiv k-l\equiv1$ or $m-1 \pmod{m}$, and $i, j, k, l$ are distinct.
\end{enumerate}

If Claim (1) is not true, without loss of generality, assume $(\lambda_{11}, \cdots, \lambda_{41})=(\lambda_{12}, \cdots, \lambda_{42})$, then for each $i$,
$$\lambda_i=\lambda_{i1}(x_1+x_2)+\lambda_{i3}x_3+\cdots+\lambda_{im}x_m.$$
Since $(x_1+x_2)^2=0$ in $\mathcal{E}_{\Bbb R}(C^4(m))$, the monomials in $\sigma_4^{C^4(m)}(x_1, ..., x_m)$ containing $x_1x_2$ must vanish. This is contradictory.
If Claim (2) is not true, also assume $(\lambda_{11}, \cdots, \lambda_{41})+(\lambda_{12}, \cdots, \lambda_{42})=(\lambda_{1k}, \cdots, \lambda_{4k})$, $k\ne 1, 2$, then
$$\lambda_i=\lambda_{i1}(x_1+x_k)+\lambda_{i2}(x_2+x_k)+\lambda_{i3}x_3+\cdots+\lambda_{im}x_m.$$
The monomials containing $x_1x_2x_k$ must be obtained from multiplying by factor $(x_1+x_k)^2$ or $(x_2+x_k)^2$, which vanishes. It is similar to verify Claim (3).

\vskip .1cm
So by the above three claims, the two sets $\{(\lambda_{1i}, \cdots, \lambda_{4i}): i=1, \cdots, m\}$ and $$\{(\lambda_{1i}, \cdots, \lambda_{4i})+(\lambda_{1j}, \cdots, \lambda_{4j}): i-j\equiv1\ \text{or}\ m-1\pmod{m}\}$$ occupy $2m$ different non-zero elements in $\Z_2^4$, which is impossible when $m\ge8$.

\vskip .1cm
 Actually, $m=8$ is the precise lower bound satisfying $s_{\Bbb R}(C^4(m))<m-4$ for the case of $n=4$. We can give the specific matrices determined by the characteristic functions of $C^4(7)$ and $C^4(8)$ as follows:
\begin{displaymath}
\left(
  \begin{array}{ccccccc}
    1 & 0 & 0 & 0 & 1 & 1 & 0 \\
    0 & 1 & 0 & 0 & 0 & 1 & 1 \\
    0 & 0 & 1 & 0 & 1 & 1 & 1 \\
    0 & 0 & 0 & 1 & 1 & 0 & 1 \\
  \end{array}
\right)
\end{displaymath} for $C^4(7)$, and
\begin{displaymath}
\left(
         \begin{array}{cccccccc}
           1 & 0 & 0 & 0 & 0 & 1 & 0 & 1 \\
           0 & 1 & 0 & 0 & 0 & 0 & 0 & 1 \\
           0 & 0 & 1 & 0 & 0 & 0 & 1 & 1 \\
           0 & 0 & 0 & 1 & 0 & 1 & 0 & 1 \\
           0 & 0 & 0 & 0 & 1 & 0 & 1 & 1 \\
         \end{array}
\right)
\end{displaymath} for $C^4(8)$, respectively. Note that the above matrix for $C^4(7)$ also means that
\begin{align*}\sigma_4^{C^4(7)}(x_1, ..., x_7)=(x_1+x_5+x_6)(x_2+x_6+x_7)(x_3+x_5+x_6+x_7)(x_4+x_5+x_7)\end{align*} in
 $\mathcal{E}_{\Bbb R}(C^4(7))$. Thus,  $s_{\Bbb R}(C^4(7))=s_{\Bbb R}(C^4(8))=3$.
\end{proof}

In the general case, we have:

\begin{proposition}
Let $n\geq 4$.
When $m\geq 2^{n-1}$, $s_{\Bbb R}(C^n(m))<m-n$.
\end{proposition}
\begin{proof}
With the above notation, as shown in  \cite{s}, the criterion for the $3$-faces of $C^n(m)$ with $n\geq 4$ is compatible. This means that  $a=\{v_{i_1}, v_{i_2}, v_{i_3}, v_{i_4}\}$ with $i_1<\cdots< i_4$ satisfies either of the following conditions:
\begin{enumerate}
  \item $i_1=1$, $i_4=m$, $i_3-i_2=1$;
  \item $i_2-i_1=i_4-i_3=1$,
\end{enumerate}
if and only if $a$ is a $3$-face in $C^n(m)$.
So the argument of Lemma~\ref{c4} can be still applied to the case of higher dimensions.
\end{proof}

\begin{corollary}
Let $n\geq 4$.
If $m\geq 2^{n-1}$, then $s_{\Bbb C}(C^n(m))<m-n$.
\end{corollary}


\end{document}